\documentclass[12pt]{amsart}
\usepackage{amssymb}
\usepackage[all]{xy}

\newtheorem{theorem}{Theorem}[section]
\newtheorem{definition}[theorem]{Definition}
\newtheorem{proposition}[theorem]{Proposition}
\newtheorem{lemma}[theorem]{Lemma}

\begin{document}

\title{Liberations and twists of real and complex spheres}

\author{Teodor Banica}
\address{T.B.: Department of Mathematics, Cergy-Pontoise University, 95000 Cergy-Pontoise, France. {\tt teodor.banica@u-cergy.fr}}


\begin{abstract}
We study the 10 noncommutative spheres obtained by liberating, twisting, and liberating+twisting the real and complex spheres $S^{N-1}_\mathbb R,S^{N-1}_\mathbb C$. At the axiomatic level, we show that, under very strong axioms, these 10 spheres are the only ones. Our main results concern the computation of the quantum isometry groups of these 10 spheres, taken in an affine real/complex sense. We formulate as well a proposal for an extended formalism, comprising 18 spheres.
\end{abstract}

\maketitle

\section*{Introduction}

A remarkable discovery, due to Goswami \cite{go1}, is that each noncommutative compact Riemannian manifold $X$ in the sense of Connes \cite{co1}, \cite{co2}, \cite{co3} has a quantum isometry group $G^+(X)$. While the classical, connected manifolds cannot have genuine quantum isometries \cite{gjo}, for the non-classical or non-connected manifolds the quantum isometry group $G^+(X)$ can be bigger than the usual isometry group $G(X)$, containing therefore ``non-classical'' symmetries, worth to be investigated.

As a motivating example, the symmetries of the finite noncommutative manifold coming from the Standard Model, axiomatized by Chamseddine and Connes in \cite{cc1}, \cite{cc2}, were studied by Bhowmick, D'Andrea, Dabrowski and Das in \cite{bdd}, \cite{bd+}. One of their findings is that the usual gauge group component $PU_3$ becomes replaced in this way by the quantum group $PU_3^+=PO_3^+=\bar{S}_9^+$. Here $O_N^+,U_N^+,S_N^+$ are the quantum groups constructed by Wang in \cite{wa1}, \cite{wa2}, and the twisting result $PO_3^+=\bar{S}_9^+$ comes from \cite{bbs}.

At a theoretical level, one interesting question is about adapting the various classical computations of isometry groups. Perhaps the most basic such computation is $G(S^{N-1}_\mathbb R)=O_N$, where $S^{N-1}_\mathbb R\subset\mathbb R^N$ is the standard sphere. Yet another standard computation, this time in the disconnected manifold case, is $G(X_N)=S_N$, where $X_N=\{e_1,\ldots,e_N\}\subset\mathbb R^N$ is the simplex, with $e_1,\ldots,e_N$ being the standard basis vectors of $\mathbb R^N$. 

Such results are of course quite trivial, but their noncommutative extensions, not always. In the discrete manifold case we have $G^+(X_N)=S_N^+$, but more complicated computations, such as $G(Y_N)=H_N$, where $Y_N=\{\pm e_1\ldots\pm e_N\}\subset\mathbb R^N$ is the hypercube, and $H_N=\mathbb Z_2\wr S_N$, lead to some interesting questions. See \cite{bbc}, \cite{go2}.

In the continuous manifold case, which is the one that we are interested in here, the extensions of the basic computation $G(S^{N-1}_\mathbb R)=O_N$ lead to interesting questions as well. This is well-known for instance in the context of the Podle\'s spheres \cite{pod}, and we refer here to \cite{bg3}, \cite{dd+}, \cite{var}. More advanced examples of noncommutative spheres, having more intricate algebraic and differential geometry, come from  \cite{cdu}, \cite{cla}.

In our joint work with Goswami \cite{bgo} we introduced two basic generalizations of $S^{N-1}_\mathbb R$, namely the half-liberated sphere $S^{N-1}_{\mathbb R,*}$, and the free sphere $S^{N-1}_{\mathbb R,+}$. These spheres appear by definition as dual objects to certain universal $C^*$-algebras, inspired by the easy quantum group philosophy \cite{bsp}. More precisely, the surjections at the $C^*$-algebra level produce inclusions $S^{N-1}_\mathbb R\subset S^{N-1}_{\mathbb R,*}\subset S^{N-1}_{\mathbb R,+}$, which are related, via the quantum isometry group construction, to the basic inclusions $O_N\subset O_N^*\subset O_N^+$ from \cite{bsp}, \cite{bve}.

Our purpose here is three-fold:
\begin{enumerate}
\item We will review the work in \cite{bgo}, with a new axiomatization of these 3 spheres, less relying on the structure of the corresponding quantum isometry groups.

\item We will present a unitary extension of \cite{bgo}, based on $G(S^{N-1}_\mathbb C)=U_N$, with the isometry group being taken in an affine complex sense.

\item We will present as well a twisting extension of \cite{bgo}, in both the real and complex cases, involving the group $\bar{O}_N$ from \cite{bbc}, and a number of related objects. 
\end{enumerate}

We will construct in this way 10 noncommutative spheres, as follows:
$$\xymatrix@R=15mm@C=15mm{
S^{N-1}_\mathbb C\ar[r]&S^{N-1}_{\mathbb C,**}\ar[r]&S^{N-1}_{\mathbb C,+}&\bar{S}^{N-1}_{\mathbb C,**}\ar@.[l]&\bar{S}^{N-1}_\mathbb C\ar@.[l]\\
S^{N-1}_\mathbb R\ar[r]\ar[u]&S^{N-1}_{\mathbb R,*}\ar[r]\ar[u]&S^{N-1}_{\mathbb R,+}\ar[u]&\bar{S}^{N-1}_{\mathbb R,*}\ar@.[u]\ar@.[l]&\bar{S}^{N-1}_\mathbb R\ar@.[u]\ar@.[l]}$$

Here all the maps are inclusions. The spheres in \cite{bgo} are those at bottom left, their complex analogues are on top left, and the whole right part of the diagram appears from the left part via twisting, with the middle spheres being equal to their own twists.

We will prove then that the associated quantum isometry groups, taken in an affine real/complex sense, in the spirit of \cite{go2}, are as follows:
$$\xymatrix@R=17mm@C=17mm{
U_N\ar[r]&U_N^{**}\ar[r]&U_N^+&\bar{U}_N^{**}\ar@.[l]&\bar{U}_N\ar@.[l]\\
O_N\ar[r]\ar[u]&O_N^*\ar[r]\ar[u]&O_N^+\ar[u]&\bar{O}_N^*\ar@.[u]\ar@.[l]&\bar{O}_N\ar@.[u]\ar@.[l]}$$

We believe that our 10 spheres are ``smooth'' and ``Riemannian'', in some strong sense, which is yet to be determined. Some questions here, still open, were raised in \cite{bgo}.

At the axiomatic level, we will have results and conjectures stating that, under very strong axioms, our 10 spheres (or ``geometries'', in a large sense) are the only ones. Our axioms exclude however many interesting objects, like the half-liberated geometry $\mathbb C^N_*$ from \cite{bdd}. Our third contribution will be a proposal, in order to fix this problem. We will show that the 10-geometry formalism has a natural 18-geometry extension, as follows:
$$
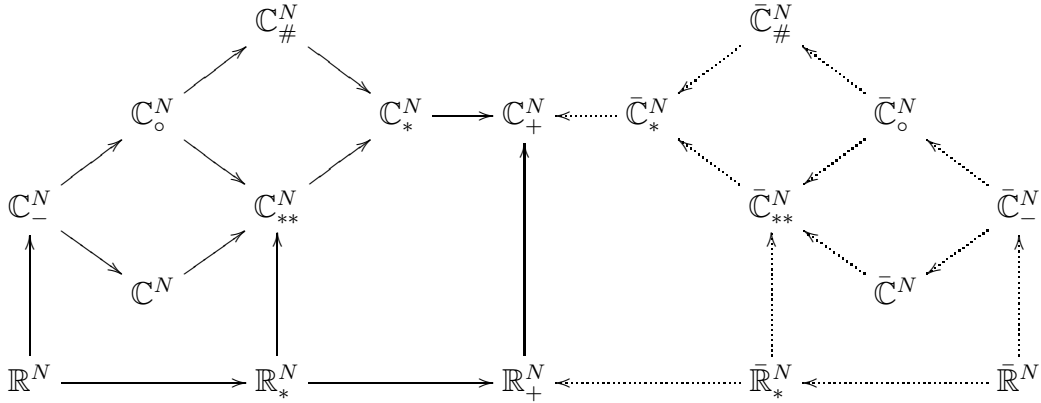
\xymatrix@R=5mm@C=8mm{
&&\mathbb C^N_\#\ar[dr]&&&&\bar{\mathbb C}^N_\#\ar@.[dl]&\\
&\mathbb C^N_\circ\ar[dr]\ar[ur]&&\mathbb C^N_*\ar[r]&\mathbb C^N_+&\bar{\mathbb C}^N_*\ar@.[l]&&\bar{\mathbb C}^N_\circ\ar@.[ul]\ar@.[dl]\\
\mathbb C^N_-\ar[ur]\ar[dr]&&\mathbb C^N_{**}\ar[ur]&&&&\bar{\mathbb C}^N_{**}\ar@.[ul]&&\bar{\mathbb C}^N_-\ar@.[ul]\ar@.[dl]\\
&\mathbb C^N\ar[ur]&&&&&&\bar{\mathbb C}^N\ar@.[ul]\\
\mathbb R^N\ar[uu]\ar[rr]&&\mathbb R^N_*\ar[rr]\ar[uu]&&\mathbb R^N_+\ar[uuu]&&\bar{\mathbb R}^N_*\ar@.[ll]\ar@.[uu]&&\bar{\mathbb R}^N\ar@.[ll]\ar@.[uu]}$$

Here the geometries $\mathbb C^N_-\to\mathbb C^N_\circ\to\mathbb C^N_\#$ and $\bar{\mathbb C}^N_-\to\bar{\mathbb C}^N_\circ\to\bar{\mathbb C}^N_\#$ are new, and appear when inserting the geometry $\mathbb C^N_*$ from \cite{bdd} and its twist into the 10-geometry framework. This extension, however, requires a lot of work, and we have only partial results here.

We refer to the body of the paper for the precise statements of our results, and to the final section below for a summary of questions raised by the present work.

The paper is organized as follows: in 1-2 we construct and axiomatize/classify the main 10 spheres, in 3-4 we study the corresponding quantum isometry groups, and in 5-6 we state and prove our main results, and we discuss the extended formalism.

\medskip

\noindent {\bf Formalism and notations.} We use the ``noncommutative compact space'' framework coming from operator algebras. More precisely, the category of noncommutative compact spaces is by definition the category of unital $C^*$-algebras, with the arrows reversed. 

According to the Gelfand theorem, the category of usual compact spaces embeds covariantly into the category of noncommutative compact spaces, via $X\to C(X)$, the image is formed by the spaces coming from the commutative $C^*$-algebras, and the inverse correspondence is obtained by taking the spectrum, $X=\{\chi:C(X)\to\mathbb C\}$.

We denote such noncommutative spaces by $X,Y,Z,\ldots$, with the corresponding $C^*$-algebras being denoted $C(X),C(Y),C(Z),\ldots$ A morphism $X\to Y$ is by definition injective if the corresponding morphism $C(Y)\to C(X)$ in surjective, and vice versa.

Most of our spaces will be of algebraic geometric nature, coming in series $\{X_N|N\in\mathbb N\}$, with each $C(X_N)$ having $N$ privileged generators $x_1,\ldots,x_N$ (the ``coordinates''), subject to certain uniform relations, not depending on $N$. We will often refer to $X_N$ as the ``specialization'' of the abstract object $X=(X_N)$, at a particular $N\in\mathbb N$. 

\medskip

\noindent {\bf Acknowledgements.} I would like to thank Julien Bichon, Uwe Franz, Adam Skalski, Georges Skandalis and Roland Speicher for useful discussions. This work was partly supported by the ``Harmonia'' NCN grant 2012/06/M/ST1/00169.
 
\section{Noncommutative spheres}

We are interested in the noncommutative, undeformed analogues of $\mathbb R^N,\mathbb C^N$. At the pure algebra level, of the corresponding $*$-algebras of polynomial functions, these analogues can be introduced by ``liberating'' and ``twisting'' the various commutativity relations $ab=ba$ appearing in the following $*$-algebra presentation results:
\begin{eqnarray*}
Pol(\mathbb R^N)&=&\left<x_1,\ldots,x_N\Big|x_i=x_i^*,x_ix_j=x_jx_i\right>\\
Pol(\mathbb C^N)&=&\left<z_1,\ldots,z_N\Big|z_iz_j=z_jz_i,z_iz_j^*=z_j^*z_j\right>
\end{eqnarray*}

However, if we want to have norms on our universal $*$-algebras, we must restrict attention to compact submanifolds $X\subset\mathbb R^N,Z\subset\mathbb C^N$. And, the most natural candidates for such submanifolds are the corresponding spheres, $S^{N-1}_\mathbb R\subset\mathbb R^N$ and $S^{N-1}_\mathbb C\subset\mathbb C^N$. 

Looking at spheres is in fact not very restrictive, because many interesting manifolds appear as $X\subset S^{N-1}_\mathbb R,Z\subset S^{N-1}_\mathbb C$. For instance, after a $1/\sqrt{N}$ rescaling of the coordinates, any compact Lie group appears as $G\subset U_N\subset S^{N^2-1}_\mathbb C$. In addition, many homogeneous spaces $G\to X$ appear as well naturally as submanifolds of spheres.

To summarize this discussion, we are interested in the noncommutative, undeformed analogues of $S^{N-1}_\mathbb R,S^{N-1}_\mathbb C$. Our starting point will be the following result:

\begin{proposition}
The algebras of continuous functions on $S^{N-1}_\mathbb R,S^{N-1}_\mathbb C$ are given by
\begin{eqnarray*}
C(S^{N-1}_\mathbb R)&=&C^*\left(x_1,\ldots,x_N\Big|x_i=x_i^*,x_ix_j=x_jx_i,\sum_ix_i^2=1\right)\\
C(S^{N-1}_\mathbb C)&=&C^*\left(z_1,\ldots,z_N\Big|z_iz_j=z_jz_i,z_iz_j^*=z_j^*z_i,\sum_iz_iz_i^*=1\right)
\end{eqnarray*}
where at right we have universal $C^*$-algebras.
\end{proposition}

\begin{proof}
This is a well-known consequence of the Gelfand and Stone-Weierstrass theorems. Indeed, the univeral algebras on the right being commutative, they are of the form $C(X),C(Z)$. The coordinate functions $x_i,z_i$ provide us with embeddings $X\subset\mathbb R^N,Z\subset\mathbb C^N$, and then the quadratic conditions give $X=S^{N-1}_\mathbb R,Z=S^{N-1}_\mathbb C$, as claimed.
\end{proof}

The idea now is to replace the commutation relations $ab=ba$ between the standard coordinates by some well-chosen relations. A first choice is that of using the anticommutation relations $ab=-ba$. A second choice, coming from the easy quantum group philosophy \cite{bsp}, is that of using the half-commutation relations $abc=cba$. A third choice, coming from the general liberation philosophy in free probability \cite{bsp}, \cite{bpa}, \cite{nsp}, \cite{vdn}, and which is perhaps the most straightforward, is that of using no relations at all.

So, let us first construct the free analogues of $S^{N-1}_\mathbb R,S^{N-1}_\mathbb C$:

\begin{definition}
The free versions of $S^{N-1}_\mathbb R,S^{N-1}_\mathbb C$ are defined by
\begin{eqnarray*}
C(S^{N-1}_{\mathbb R,+})&=&C^*\left(x_1,\ldots,x_N\Big|x_i=x_i^*,\sum_ix_i^2=1\right)\\
C(S^{N-1}_{\mathbb C,+})&=&C^*\left(z_1,\ldots,z_N\Big|\sum_iz_iz_i^*=\sum_iz_i^*z_i=1\right)
\end{eqnarray*}
where at right we have universal $C^*$-algebras.
\end{definition}

Here the fact that the norms are bounded, and hence that the above universal algebras do exist, comes from the quadratic conditions, which give $||x_i||\leq1,||z_i||\leq1$. 

Observe that our definition of $S^{N-1}_{\mathbb C,+}$ involves both the equalities $\sum_iz_iz_i^*=1$ and $\sum_iz_i^*z_i=1$, instead of just a single one. There are several reasons for this choice: 
\begin{enumerate}
\item We would like, as in usual projective geometry, the matrix $p=(p_{ij})$ formed by the elements $p_{ij}=z_iz_j^*$ to satisfy $p=p^*=p^2,Tr(p)=1$. And, the verification of these conditions requires both $\sum_iz_iz_i^*=1$ and $\sum_iz_i^*z_i=1$.

\item We would like as well, once again in analogy with the classical case, the generators $z_i$ to satisfy same the algebraic relations as the variables $\gamma_i=u_{1i}$ over the quantum group $U_N^+$. And, these latter variables satisfy $\sum_i\gamma_i\gamma_i^*=\sum_i\gamma_i^*\gamma_i=1$.
\end{enumerate}

We will be back later on to these topics, with concrete results justifying our choice, and with some axiomatization results as well, once again relying on this choice.

Let us construct now the twisted versions of $S^{N-1}_\mathbb R,S^{N-1}_\mathbb C$. In our generators and relations framework, these two spheres are best introduced as follows:

\begin{definition}
The twisted versions of $S^{N-1}_\mathbb R,S^{N-1}_\mathbb C$ are defined by
\begin{eqnarray*}
C(\bar{S}^{N-1}_\mathbb R)&=&C(S^{N-1}_{\mathbb R,+})\Big/\left<ab=-ba,\forall a,b\in\{x_i\}\ {\rm distinct}\right>\\
C(\bar{S}^{N-1}_\mathbb C)&=&C(S^{N-1}_{\mathbb C,+})\Big/\left<\alpha\beta=-\beta\alpha,\forall a,b\in\{z_i\}\ {\rm distinct},\ \alpha\beta=\beta\alpha\ {\rm otherwise}\right>
\end{eqnarray*}
where we use the notations $\alpha=a,a^*$ and $\beta=b,b^*$.
\end{definition}

In other words, the defining relations for $\bar{S}^{N-1}_\mathbb R$ are $x_ix_j=-x_jx_i$ for any $i\neq j$, and those for $\bar{S}^{N-1}_\mathbb C$ are $z_iz_i^*=z_i^*z_i$ for any $i$, and $z_iz_j=-z_jz_i,z_iz_j^*=-z_j^*z_i$ for any $i\neq j$.

Regarding the free spheres in Definition 1.2, these cannot be twisted. This is well-known, and we will use the conventions $\bar{S}^{N-1}_{\mathbb R,+}=S^{N-1}_{\mathbb R,+}$, $\bar{S}^{N-1}_{\mathbb C,+}=S^{N-1}_{\mathbb C,+}$, where needed.

Let us discuss now the half-liberation operation. In the real case this is obtained by using the relations $abc=cba$. In the complex case there are several choices, as explained in \cite{bdd}, \cite{bdu}. We will use here the ``minimal'' choice, from \cite{bdu}. The other choices, including the ``maximal'' one from \cite{bdd}, will be discussed later on.

So, let us construct four more spheres, as follows:

\begin{definition}
The half-liberations of $S^{N-1}_\mathbb R,S^{N-1}_\mathbb C,\bar{S}^{N-1}_\mathbb R,\bar{S}^{N-1}_\mathbb C$ are defined by
\begin{eqnarray*}
C(S^{N-1}_{\mathbb R,*})&=&C(S^{N-1}_{\mathbb R,+})\Big/\left<abc=cba,\forall a,b,c\in\{x_i\}\right>\\
C(S^{N-1}_{\mathbb C,**})&=&C(S^{N-1}_{\mathbb C,+})\Big/\left<abc=cba,\forall a,b,c\in\{z_i,z_i^*\}\right>\\
C(\bar{S}^{N-1}_{\mathbb R,*})&=&C(S^{N-1}_{\mathbb R,+})\Big/\left<abc=-cba,\forall a,b,c\in\{x_i\}\ {\rm distinct},abc=cba\ {\rm otherwise}\right>\\
C(\bar{S}^{N-1}_{\mathbb C,**})&=&C(S^{N-1}_{\mathbb C,+})\Big/\left<\alpha\beta\gamma=-\gamma\beta\alpha,\forall a,b,c\in\{z_i\}\ {\rm distinct},\ \alpha\beta\gamma=\gamma\beta\alpha\ {\rm otherwise}\right>
\end{eqnarray*}
where we use the notations $\alpha=a,a^*$, $\beta=b,b^*$ and $\gamma=c,c^*$.
\end{definition}

We have so far $2+2+2+4=10$ spheres, and we will temporarily stop here, because we will see in the next section that, under strong axioms, these spheres are the only ones. We will be back to more complicated examples later on, in section 6 below. 

As a first result about these 10 spheres, we have:

\begin{proposition}
We have the following diagram,
$$\xymatrix@R=15mm@C=15mm{
S^{N-1}_\mathbb C\ar[r]&S^{N-1}_{\mathbb C,**}\ar[r]&S^{N-1}_{\mathbb C,+}&\bar{S}^{N-1}_{\mathbb C,**}\ar@.[l]&\bar{S}^{N-1}_\mathbb C\ar@.[l]\\
S^{N-1}_\mathbb R\ar[r]\ar[u]&S^{N-1}_{\mathbb R,*}\ar[r]\ar[u]&S^{N-1}_{\mathbb R,+}\ar[u]&\bar{S}^{N-1}_{\mathbb R,*}\ar@.[u]\ar@.[l]&\bar{S}^{N-1}_\mathbb R\ar@.[u]\ar@.[l]}$$
with all the maps being inclusions.
\end{proposition}

\begin{proof}
In the untwisted case all the inclusions are clear from definitions. In the twisted case most of the inclusions are clear too, and we just have to check the two horizontal inclusions at right. Regarding the inclusion $\bar{S}^{N-1}_\mathbb R\subset\bar{S}^{N-1}_{\mathbb R,*}$, here the statement is that $ab=-ba$ for $a\neq b$ implies $abc=-cba$ for $a,b,c$ distinct, and $abc=cba$ otherwise. 

The first claim follows from $abc=-bac=bca=-cba$.

Regarding now the second claim, in the case $a=b=c$ we have $aaa=aaa$, in the case $a=b\neq c$ we have $aac=-aca=caa$, in the case $a=c\neq b$ we have $aba=aba$, and in the case $b=c\neq a$ we have $abb=-bab=bba$, and this finishes the proof.

Regarding the remaining inclusion, $\bar{S}^{N-1}_\mathbb C\subset\bar{S}^{N-1}_{\mathbb C,**}$, the proof here is similar, by replacing $a,b,c$ with variables $\alpha,\beta,\gamma$, given by $\alpha=a,a^*$, $\beta=b,b^*$ and $\gamma=c,c^*$.
\end{proof}

We investigate now the properness of the inclusions in the above diagram. A simple criterion for comparing spheres is by looking at the classical versions. We have here:

\begin{proposition}
The classical versions of the $10$ spheres are
$$\xymatrix@R=15mm@C=15mm{
S^{N-1}_\mathbb C\ar[r]&S^{N-1}_\mathbb C\ar[r]&S^{N-1}_\mathbb C&S^{N-1,1}_\mathbb C\ar@.[l]&\mathbb T^{\oplus N}\ar@.[l]\\
S^{N-1}_\mathbb R\ar[r]\ar[u]&S^{N-1}_\mathbb R\ar[r]\ar[u]&S^{N-1}_\mathbb R\ar[u]&S^{N-1,1}_\mathbb R\ar@.[u]\ar@.[l]&\mathbb Z_2^{\oplus N}\ar@.[u]\ar@.[l]}$$
where $S^{N-1,1}_\mathbb K$ is a union of $\binom{N}{2}$ copies of $S^1_\mathbb K$, which is not smooth at $N\geq3$.
\end{proposition}

\begin{proof}
The assertions for the untwisted spheres are clear by definition.

Observe that we have $S^{N-1}_\mathbb C\cap\bar{S}^{N-1}_\mathbb C=\mathbb T^{\oplus N}$, because the relations for $\bar{S}^{N-1}_\mathbb C$, applied to the points $z\in S^{N-1}_\mathbb C$, read $ab=0$, for any $a,b\in\{z_i\}$ distinct. We conclude that such points $z$ are those having all but one coordinates vanishing, $z\in\mathbb T^{\oplus N}$. 

By restricting now to the real case, we obtain $S^{N-1}_\mathbb R\cap\bar{S}^{N-1}_\mathbb R=\mathbb Z_2^{\oplus N}$ as well.

Regarding the intersections $S^{N-1,1}_\mathbb R=S^{N-1}_\mathbb R\cap\bar{S}^{N-1}_{\mathbb R,*}$ and $S^{N-1,1}_\mathbb C=S^{N-1}_\mathbb C\cap\bar{S}^{N-1}_{\mathbb C,**}$, observe that a point $z\in S^{N-1}_\mathbb K$ belongs to $S^{N-1,1}_\mathbb K$ precisely when its coordinates satisfy $z_iz_jz_l=0$, for any $i,j,l$ distinct. Thus $S^{N-1,1}_\mathbb K$ is the union of $\binom{N}{2}$ copies of $S^1_\mathbb K$, as claimed. 

Finally, the non-smoothness assertion is clear.
\end{proof}

Now back to the properness question, we have here:

\begin{theorem}
The inclusions in Proposition 1.5 are as follows:
\begin{enumerate}
\item At $N\geq3$, all these inclusions are proper.

\item At $N=2$ we have $S^1_{\mathbb R,*}=\bar{S}^1_{\mathbb R,*}=S^1_{\mathbb R,+}$, and the other inclusions are proper.
\end{enumerate}
\end{theorem}

\begin{proof}
We first discuss the general case, $N\geq2$. Here the 5 vertical inclusions are all proper, by Proposition 1.6. In order to check that the 4 horizontal inclusions at left and at right are proper, we can use a trick from \cite{bdu}. Consider indeed one of the spheres $S^{N-1}_\mathbb C/\bar{S}^{N-1}_\mathbb C$, with coordinates denoted $z_1,\ldots,z_N$, and let us set:
$$X_i=\begin{pmatrix}0&z_i\\ z_i^*&0\end{pmatrix}$$

These matrices are self-adjoint, they half-commute/half-anticommute, and their squares sum up to 1, so they produce a representation of  $S^{N-1}_{\mathbb R,*}/\bar{S}^{N-1}_{\mathbb R,*}$. Now since these matrices don't commute/anticommute, $S^{N-1}_\mathbb R\subset S^{N-1}_{\mathbb R,*}$ and $\bar{S}^{N-1}_\mathbb R\subset\bar{S}^{N-1}_{\mathbb R,*}$ are proper.

It follows that the inclusions $S^{N-1}_\mathbb C\subset S^{N-1}_{\mathbb C,**}$ and $\bar{S}^{N-1}_\mathbb C\subset\bar{S}^{N-1}_{\mathbb C,**}$ are proper as well, because these inclusions appear from the above ones, by intersecting with $S^{N-1}_{\mathbb R_+}$.

It remains to investigate the 4 horizontal inclusions in the middle:

(1) \underline{Case $N\geq 3$}. By intersecting everything with $S^2_{\mathbb R,+}$, it is enough to prove that the inclusions $S^2_{\mathbb R,*}\subset S^2_{\mathbb R,+}$ and $\bar{S}^2_{\mathbb R,*}\subset S^2_{\mathbb R,+}$ are proper. For $\bar{S}^2_{\mathbb R,*}\subset S^2_{\mathbb R,+}$, this follows from the fact that the inclusion of corresponding classical versions $S^{2,1}_\mathbb R\subset S^2_\mathbb R$ is proper.

For the inclusion $S^2_{\mathbb R,*}\subset S^2_{\mathbb R,+}$, we can use a trick from \cite{bgo}. Consider indeed the positive matrices in $M_2(\mathbb C)$, which are of the following form:
$$Y=\begin{pmatrix}r&z\\ \bar{z}&s\end{pmatrix}$$

Here $r,s\in\mathbb R$ and $z\in\mathbb C$ must be chosen such that both eigenvalues are positive, and this happens for instance when $r,s>0$ and $z\in\mathbb C$ is small enough.

Let us fix some numbers $r_i,s_i>0$ with $i=1,2,3$, satisfying $\sum_ir_i=\sum_is_i=1$. For any choice of small complex numbers $z_i\in\mathbb C$ satisfying $\sum_iz_i=0$, the corresponding elements $Y_i$ constructed as above will be positive, and will sum up to $1$. Moreover, by carefully choosing the $z_i$'s, we can arrange as for $Y_1,Y_2,Y_3$ not to pairwise commute.

Consider now the matrices $X_i=\sqrt{Y_i}$. These are all self-adjoint, and their squares sum up to 1, so we get a representation $C(S^2_{\mathbb R,+})\to M_2(\mathbb C)$ mapping $x_i\to X_i$. Observe that this representation maps $x_i^2\to Y_i$, and the elements $Y_i$ don't commute. 

Now since the relations $abc=cba$ imply $aabb=baab=bbaa$, the squares of the standard coordinates on $S^2_{\mathbb R,*}$ commute. We conclude that $S^2_{\mathbb R,*}\subset S^2_{\mathbb R,+}$ is indeed proper.

(2) \underline{Case $N=2$}. Here we must prove that, among the 4 horizontal inclusions in the middle, the two upper ones are proper, and the two lower ones are isomorphisms:
$$\xymatrix@R=15mm@C=15mm{
S^1_\mathbb C\ar[r]&S^1_{\mathbb C,**}\ar[r]&S^1_{\mathbb C,+}&\bar{S}^1_{\mathbb C,**}\ar@.[l]&\bar{S}^1_\mathbb C\ar@.[l]\\
S^1_\mathbb R\ar[r]\ar[u]&S^1_{\mathbb R,+}\ar[r]\ar[u]&S^1_{\mathbb R,+}\ar[u]&S^1_{\mathbb R,+}\ar@.[u]\ar@.[l]&\bar{S}^1_\mathbb R\ar@.[u]\ar@.[l]}$$

In order to prove $S^1_{\mathbb R,*}=\bar{S}^1_{\mathbb R,*}=S^1_{\mathbb R,+}$ observe that, since we have only two coordinates $x,y$, the half-commutation relations $abc=\pm bca$ reduce to the commutation relations $xy^2=y^2x,x^2y=yx^2$. But these relations hold over $S^1_{\mathbb R,+}$, because $x^2+y^2=1$.

It remains to prove that the inclusions $S^1_{\mathbb C,**}\subset S^1_{\mathbb C,+}$ and $\bar{S}^1_{\mathbb C,**}\subset S^1_{\mathbb C,+}$ are proper. In order to do so, we can use a free complexification trick, cf. \cite{ban}. Let indeed $z,t$ be the standard coordinates on $S^1_\mathbb C$, let $u$ be a unitary free from both $z,t$, and set $Z=uz,T=ut$. Then $ZZ^*+TT^*=Z^*Z+T^*T=1$, and since the relations $Z^2T=\pm TZ^2$ are not satisfied, we conclude that $S^1_{\mathbb C,**}\subset S^1_{\mathbb C,+}$ and $\bar{S}^1_{\mathbb C,**}\subset S^1_{\mathbb C,+}$ are both proper.
\end{proof}

Summarizing, we have constructed so far 10 basic examples of underformed noncommutative spheres. We will study in detail these spheres in sections 2-5 below, and we will come back to more complicated examples in section 6 below.

\section{Axiomatization, classification}

In this section we prove that, under a suitable axiomatization for the undeformed noncommutative spheres, the 10 spheres constructed above are the only ones.

Our axioms will be of course very strong. In order to introduce them, let us begin with some heuristics. The common features of our 10 spheres can be summarized as:

\begin{proposition}
The $10$ spheres appear from $S^{N-1}_{\mathbb C,+}$, by imposing relations of type
$$\alpha=\alpha^*,\qquad\alpha\beta=\beta\alpha,\qquad\alpha\beta=\pm\beta\alpha,\qquad\alpha\beta\gamma=\gamma\beta\alpha,\qquad\alpha\beta\gamma=\pm\gamma\beta\alpha$$
with $\alpha=a,a^*$, $\beta=b,b^*$, $\gamma=c,c^*$, and where the signs come from anticommutation.
\end{proposition}

\begin{proof}
This is clear from the definition of the 10 spheres in section 1 above, with the sign claim coming from the computations in the proof of Proposition 1.5.
\end{proof}

The point now is that the above 5 types of relations, all coming from certain permutations in $S_1,S_2,S_3$, can be represented by suitable diagrams, as follows:
$$
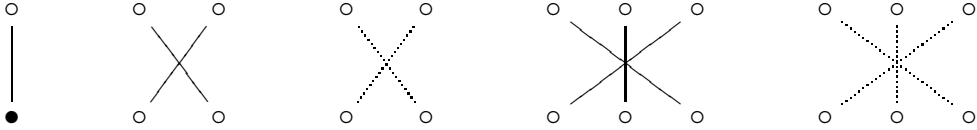
\xymatrix@R=10mm@C=6mm{\circ\ar@{-}[d]\\\bullet}\quad\qquad
\xymatrix@R=10mm@C=6mm{\circ\ar@{-}[dr]&\circ\ar@{-}[dl]\\\circ&\circ}\quad\qquad
\xymatrix@R=10mm@C=6mm{\circ\ar@{.}[dr]&\circ\ar@{.}[dl]\\\circ&\circ}\quad\qquad
\xymatrix@R=10mm@C=5mm{\circ\ar@{-}[drr]&\circ\ar@{-}[d]&\circ\ar@{-}[dll]\\\circ&\circ&\circ}\quad\qquad
\xymatrix@R=10mm@C=5mm{\circ\ar@{.}[drr]&\circ\ar@{.}[d]&\circ\ar@{.}[dll]\\\circ&\circ&\circ}$$

More precisely, associated to such a diagram is the relation obtained by putting coordinates on the legs, such as each string joins equal coordinates, and then by stating that the product on top equals the product on the bottom. And this, with the convention that the empty/full circles represent symbols of type $\alpha=a,a^*$, and their conjugates, and that the dotted diagrams bring $\pm$ signs, coming from anticommutation.

In short, we can develop a diagrammatic approach to the axiomatization problem. Before doing so, however, there are two important remarks to be made:

{\bf I}. We know from Proposition 1.6 that non-smooth manifolds can appear when intersecting twisted and untwisted spheres, and more specifically that $S^{N-1}_\mathbb C\cap\bar{S}^{N-1}_{\mathbb C,**}$ is not smooth. Thus, we do not want to mix usual diagrams with dotted diagrams:
$$
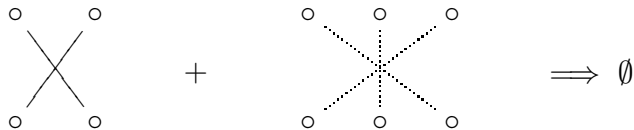
\xymatrix@R=10mm@C=6mm{\circ\ar@{-}[dr]&\circ\ar@{-}[dl]\\\circ&\circ}\qquad\xymatrix@R=4mm@C=6mm{&\\ +\\&\\& }\ \xymatrix@R=10mm@C=5mm{\circ\ar@{.}[drr]&\circ\ar@{.}[d]&\circ\ar@{.}[dll]\\\circ&\circ&\circ}\qquad\xymatrix@R=4mm@C=6mm{&\\ \implies\emptyset\\&\\& }$$
\vskip-5mm
{\bf II.} We do not want to mix either the real and complex cases. Indeed, this would amount in labelling ``black and white'' all the legs of our diagrams, and the problem is that this would produce many spheres, some of which are pathological. As an example, consider the ``sphere'' obtained from $S^{N-1}_{\mathbb C,+}$ by assuming that the coordinates $z_1,\ldots,z_N$ satisfy $ab=ba$. We would like later on this sphere to have a geometry, and a quantum isometry group. But, at the quantum group level, by using the formalism in \cite{bsp}:
$$\xymatrix@R=10mm@C=5mm{\bullet\ar@{-}[d]\ar@/^/@{-}[r]&\circ\ar@{-}[dr]&\circ\ar@{-}[dl]&\bullet\ar@{-}[d]\\\bullet&\circ&\circ\ar@/_/@{-}[r]&\bullet}\quad\xymatrix@R=4mm@C=6mm{&\\ =\\&\\& }\xymatrix@R=10mm@C=6mm{\circ\ar@{-}[dr]&\bullet\ar@{-}[dl]\\\bullet&\circ}$$
\vskip-5mm
In other words, for a unitary quantum group the relations $ab=ba$ between the standard coordinates imply the relations $ab^*=b^*a$, and so the quantum group is classical. Thus, the above ``sphere'', while being bigger than $S^{N-1}_\mathbb C$, would have the same quantum isometry group as $S^{N-1}_\mathbb C$. And this is a pathology, and so this sphere must be excluded.

Summarizing, we have to discuss separately the cases $\mathbb R,\mathbb C,\bar{\mathbb R},\bar{\mathbb C}$. Let us begin with:

\begin{definition}
Let $\dot{\mathbb K}=\mathbb R,\mathbb C,\bar{\mathbb R},\bar{\mathbb C}$ be one of the fields $\mathbb R,\mathbb C$, with the bar standing for the fact that the associated sphere is by definition the twisted one.
\begin{enumerate}
\item A monomial relation over $\dot{\mathbb K}$ is a formula of type $a_{i_1}\ldots a_{i_k}=\pm a_{i_\sigma(1)}\ldots a_{i_\sigma(k)}$, where $\sigma\in S_k$ is a permutation, and where the $\pm$ sign is the one making the formula $z_{i_1}\ldots z_{i_k}=\pm z_{i_\sigma(1)}\ldots z_{i_\sigma(k)}$ hold, over the sphere $\dot{S}^{N-1}_\mathbb K$.

\item A monomial sphere over $\dot{\mathbb K}$ is a quantum subspace $S\subset S^{N-1}_{\mathbb K,+}$ defined via a formula of type $C(S)=C(S^{N-1}_{\mathbb K,+})/<R>$, where $R$ comes from a set of monomial relations, each applied to all the variables $\gamma_i=x_i$ at $\mathbb K=\mathbb R$, and $\gamma_i=z_i,z_i^*$ at $\mathbb K=\mathbb C$.
\end{enumerate}
\end{definition}

Observe that our 10 spheres are all monomial, coming from the relations $ab=\pm ba$ and $abc=\pm cba$, corresponding to the permutations $(21)\in S_2$ and $(321)\in S_3$. Here, and in what follows, we use the permutation convention $\sigma=(\sigma(1)\ldots \sigma(k))$.

We agree to represent all permutations by diagrams, acting by definition downwards. As an example, the permutations $(21)\in S_2$ and $(321)\in S_3$ are represented as follows:
$$
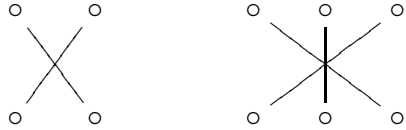
\xymatrix@R=10mm@C=6mm{\circ\ar@{-}[dr]&\circ\ar@{-}[dl]\\\circ&\circ}\qquad\qquad
\xymatrix@R=10mm@C=5mm{\circ\ar@{-}[drr]&\circ\ar@{-}[d]&\circ\ar@{-}[dll]\\\circ&\circ&\circ}$$

Observe that each monomial sphere over $\dot{\mathbb K}$ contains the sphere $\dot{S}^{N-1}_\mathbb K$, because each monomial relation is satisfied by definition by the standard coordinates of $\dot{S}^{N-1}_\mathbb K$.

The monomial spheres are best parametrized by groups, as follows:

\begin{proposition}
Given a set of permutations $E\subset S_\infty$, denote by $\dot{S}^{N-1}_{\mathbb K,E}$ the associated monomial sphere over $\dot{\mathbb K}$, with the relations $R$ coming from the elements $\sigma\in E$. Then any monomial sphere is of the form $\dot{S}^{N-1}_{\mathbb K,G}$, with $G\subset S_\infty$ being a group.
\end{proposition}

\begin{proof}
Consider indeed the set $G\subset S_\infty$ consisting of elements $\sigma\in S_\infty$ such that the relations $a_{i_1}\ldots a_{i_k}=a_{i_\sigma(1)}\ldots a_{i_\sigma(k)}$ hold, in our monomial sphere. 

It is clear then that $G$ is stable under composition, because $X=Y,Y=Z$ implies $X=Z$. Also clear is the fact that $G$ is stable under inversion, because $X=Y$ implies $Y=X$, and the fact that $G$ contains the unit permutation. Thus, $G$ is indeed a group.
\end{proof}

As an illustration for this result, by using the convention $*=**$, in order to denote the half-liberation operation by $*$ in both the real and complex cases, we have:

\begin{proposition}
The basic spheres $\dot{S}^{N-1}_\mathbb K\subset\dot{S}^{N-1}_{\mathbb K,*}\subset S^{N-1}_{\mathbb K,+}$ come from the groups 
$$S_\infty\supset S_\infty^*\supset\{1\}$$
where $S_\infty^*=(S_n^*)_{n\geq1}$ is such that $S_{2n}^*\simeq S_n\times S_n$, $S^*_{2n+1}\simeq S_n\times S_{n+1}$.
\end{proposition}

\begin{proof}
The assertions regarding $\dot{S}^{N-1}_\mathbb K,S^{N-1}_{\mathbb K,+}$ are trivial. Regarding now $\dot{S}^{N-1}_{\mathbb K,*}$, the result being insensitive to the value of $\dot{\mathbb K}$, we can assume that we are dealing with $S^{N-1}_{\mathbb R,*}$.

We use the fact, from \cite{bve}, that the relations $abc=cba$ imply the relations of type $a_{i_1}\ldots a_{i_k}=a_{i_\sigma(1)}\ldots a_{i_\sigma(k)}$, for any $\sigma\in S_k$ having the property that when labelling cyclically the legs $\bullet\circ\bullet\circ\ldots$, each string joins a black leg to a white leg. In addition, these relations imply the original relations $abc=cba$, because the permutation $(321)\in S_3$ implementing these relations has indeed the ``black-to-white'' joining property:
$$
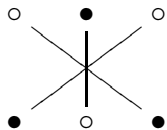
\xymatrix@R=10mm@C=5mm{\circ\ar@{-}[drr]&\bullet\ar@{-}[d]&\circ\ar@{-}[dll]\\\bullet&\circ&\bullet}$$

We conclude that $S^{N-1}_{\mathbb R,*}$ comes from the group $S_\infty^*$ consisting of permutations $\sigma\in S_\infty$ having the black-to-white joining property. Now observe that $S_3^*,S_4^*$ are given by:
$$
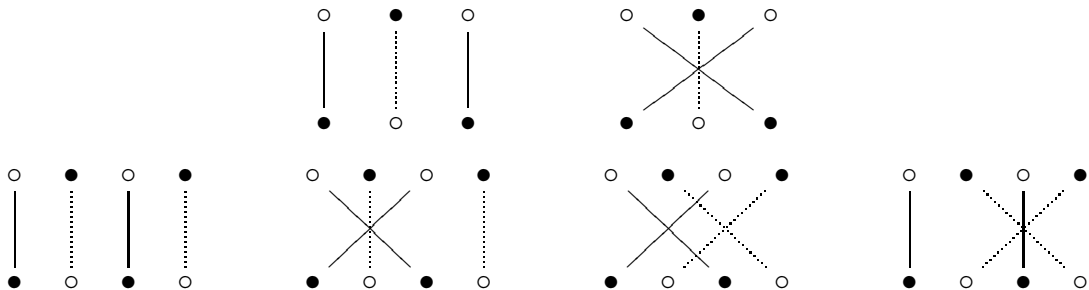
\xymatrix@R=10mm@C=5mm{\circ\ar@{-}[d]&\bullet\ar@{.}[d]&\circ\ar@{-}[d]\\\bullet&\circ&\bullet}\qquad\qquad 
\xymatrix@R=10mm@C=5mm{\circ\ar@{-}[drr]&\bullet\ar@{.}[d]&\circ\ar@{-}[dll]\\\bullet&\circ&\bullet}$$ 
$$\xymatrix@R=10mm@C=3mm{\circ\ar@{-}[d]&\bullet\ar@{.}[d]&\circ\ar@{-}[d]&\bullet\ar@{.}[d]\\\bullet&\circ&\bullet&\circ}\quad\qquad
\xymatrix@R=10mm@C=3mm{\circ\ar@{-}[drr]&\bullet\ar@{.}[d]&\circ\ar@{-}[dll]&\bullet\ar@{.}[d]\\\bullet&\circ&\bullet&\circ}\quad\qquad
\xymatrix@R=10mm@C=3mm{\circ\ar@{-}[drr]&\bullet\ar@{.}[drr]&\circ\ar@{-}[dll]&\bullet\ar@{.}[dll]\\\bullet&\circ&\bullet&\circ}\quad\qquad
\xymatrix@R=10mm@C=3mm{\circ\ar@{-}[d]&\bullet\ar@{.}[drr]&\circ\ar@{-}[d]&\bullet\ar@{.}[dll]\\\bullet&\circ&\bullet&\circ}$$

Thus we have $S_3^*=S_1\times S_2$ and $S_4^*=S_2\times S_2$, with the first component of each product coming from dotted permutations, and with the second component coming from the solid line permutations. In the general case, the proof is similar.
\end{proof}

We call depth of a monomial sphere the smallest $k\in\mathbb N\cup\{\infty\}$ such that our sphere can be written as $\dot{S}^{N-1}_{\mathbb K,E}$, as in Proposition 2.3, with $E\subset S_k$. In other words, a monomial sphere is of depth $\leq k$ when the relations defining it come from permutations $\sigma\in S_k$.

With this convention, we have the following result:

\begin{theorem}
The $10$ fundamental spheres, which can be written as 
$$\xymatrix@R=15mm@C=15mm{
S^{N-1}_{\mathbb C,S_\infty}\ar[r]&S^{N-1}_{\mathbb C,S_\infty^*}\ar[r]&S^{N-1}_{\mathbb C,\{1\}}&\bar{S}^{N-1}_{\mathbb C,S_\infty^*}\ar@.[l]&\bar{S}^{N-1}_{\mathbb C,S_\infty}\ar@.[l]\\
S^{N-1}_{\mathbb R,S_\infty}\ar[r]\ar[u]&S^{N-1}_{\mathbb R,S_\infty^*}\ar[r]\ar[u]&S^{N-1}_{\mathbb R,\{1\}}\ar[u]&\bar{S}^{N-1}_{\mathbb R,S_\infty^*}\ar@.[u]\ar@.[l]&\bar{S}^{N-1}_{\mathbb R,S_\infty}\ar@.[u]\ar@.[l]}$$
are precisely the monomial spheres having depth $k\leq 3$.
\end{theorem}

\begin{proof}
The first assertion follows from Proposition 2.4. In order to prove the uniqueness, we have to examine the 6 elements of $S_3$. These are as follows:
$$
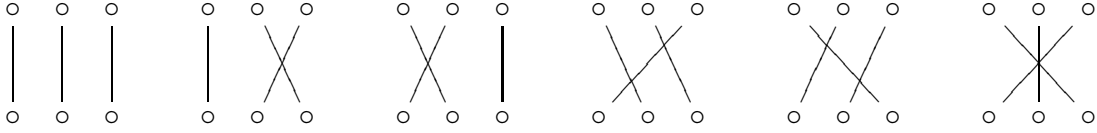
\xymatrix@R=10mm@C=2mm{\circ\ar@{-}[d]&\circ\ar@{-}[d]&\circ\ar@{-}[d]\\\circ&\circ&\circ}\qquad
\xymatrix@R=10mm@C=2mm{\circ\ar@{-}[d]&\circ\ar@{-}[dr]&\circ\ar@{-}[dl]\\\circ&\circ&\circ}\qquad
\xymatrix@R=10mm@C=2mm{\circ\ar@{-}[dr]&\circ\ar@{-}[dl]&\circ\ar@{-}[d]\\\circ&\circ&\circ}\qquad
\xymatrix@R=10mm@C=2mm{\circ\ar@{-}[dr]&\circ\ar@{-}[dr]&\circ\ar@{-}[dll]\\\circ&\circ&\circ}\qquad
\xymatrix@R=10mm@C=2mm{\circ\ar@{-}[drr]&\circ\ar@{-}[dl]&\circ\ar@{-}[dl]\\\circ&\circ&\circ}\qquad
\xymatrix@R=10mm@C=2mm{\circ\ar@{-}[drr]&\circ\ar@{-}[d]&\circ\ar@{-}[dll]\\\circ&\circ&\circ}$$

According to our diagrammatic conventions, the identity produces the 2 free spheres, the basic crossing, which appears twice, produces the 4 classical + twisted spheres, and the last diagram produces the 4 half-liberated spheres. Our claim now, which will finish the proof, is that the 3-cycles produce the same spheres as the basic crossing.

Let us first discuss the case $\dot{\mathbb K}=\mathbb R$. Here the 3-cycle produce the ``sphere'' given by $abc=cab$. The point now is that, by using these relations, we obtain:
\begin{eqnarray*}
(ab-ba)^2
&=&abab-abba-baab+baba\\
&=&aabb-aabb-aabb+baab\\
&=&aabb-aabb-aabb+aabb\\
&=&0
\end{eqnarray*}

Thus the sphere collapses to $S^{N-1}_\mathbb R$, and we are done.

In the case $\dot{\mathbb K}=\bar{\mathbb R}$, the proof is similar. Indeed, the 3-cycle produces relations of type $abc=\pm cab$, the precise formulae being: (1) $abc=-acb=cab$ for $a,b,c$ distinct, (2) $aac=-aca=caa$ for $a\neq c$, (3) $aba=-aab$ for $a\neq b$, (4) $abb=-bab$ for $a\neq c$.

With these relations in hand, we have the following computation:
\begin{eqnarray*}
(ab+ba)^2
&=&abab+abba+baab+baba\\
&=&-aabb+aabb+aabb-baab\\
&=&-aabb+aabb+aabb-aabb\\
&=&0
\end{eqnarray*}

Thus the sphere collapses to $\bar{S}^{N-1}_{\mathbb R}$, and we are done.

Finally, in the remaining two cases $\dot{\mathbb K}=\mathbb C,\bar{\mathbb C}$ the proof of the extra needed formula, namely $ab^*=\pm b^*a$, is similar, by adding $*$ exponents where needed.
\end{proof}

The above result is complementary to those in \cite{bgo}. Let us recall indeed from there that the spheres $S^{N-1}_\mathbb R\subset S^{N-1}_{\mathbb R,*}\subset S^{N-1}_{\mathbb R,+}$ are precisely those whose corresponding quantum isometry group is easy. This is of course quite a sophisticated result, and Theorem 2.5 above, formulated directly in terms of the spheres themselves, is in a certain sense ``better''. However, unifying Theorem 2.5 with \cite{bgo} remains an open question. 

Let us discuss now what happens at depth 4:

\begin{proposition}
There are no new monomial spheres at depth $4$.
\end{proposition}

\begin{proof}
We must study the 24 elements of $S_4$. These are as follows:
$$
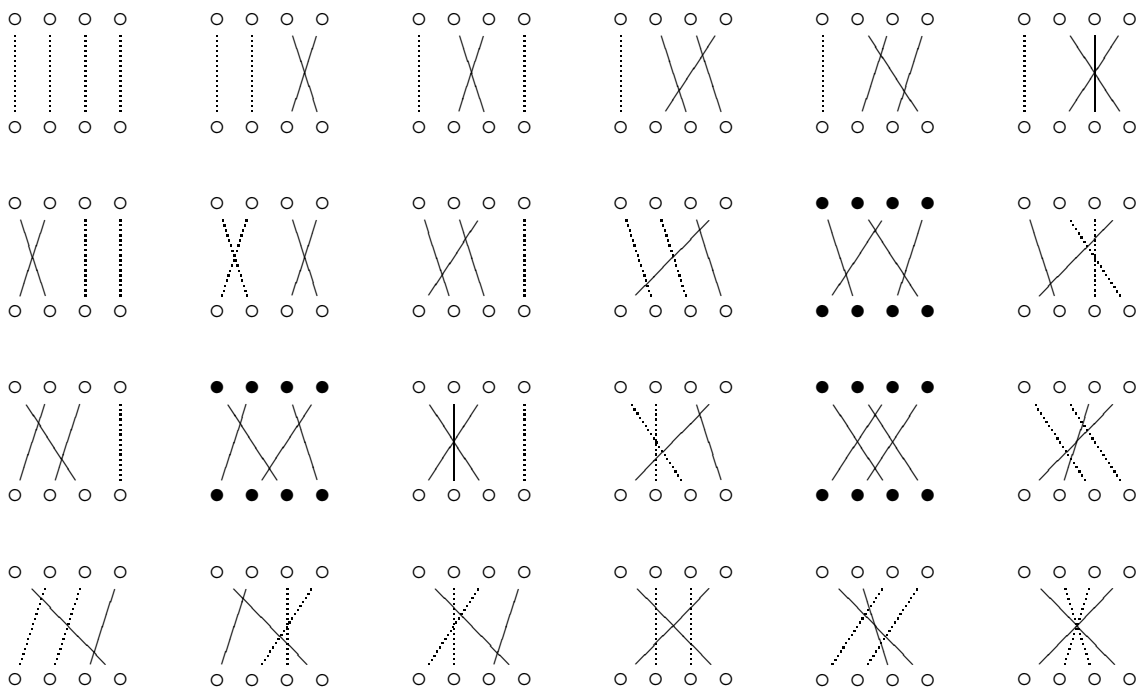
\xymatrix@R=10mm@C=0.1mm{\circ\ar@{.}[d]&\circ\ar@{.}[d]&\circ\ar@{.}[d]&\circ\ar@{.}[d]\\\circ&\circ&\circ&\circ}\qquad
\xymatrix@R=10mm@C=0.1mm{\circ\ar@{.}[d]&\circ\ar@{.}[d]&\circ\ar@{-}[dr]&\circ\ar@{-}[dl]\\\circ&\circ&\circ&\circ}\qquad
\xymatrix@R=10mm@C=0.1mm{\circ\ar@{.}[d]&\circ\ar@{-}[dr]&\circ\ar@{-}[dl]&\circ\ar@{.}[d]\\\circ&\circ&\circ&\circ}\qquad 
\xymatrix@R=10mm@C=0.1mm{\circ\ar@{.}[d]&\circ\ar@{-}[dr]&\circ\ar@{-}[dr]&\circ\ar@{-}[dll]\\\circ&\circ&\circ&\circ}\qquad
\xymatrix@R=10mm@C=0.1mm{\circ\ar@{.}[d]&\circ\ar@{-}[drr]&\circ\ar@{-}[dl]&\circ\ar@{-}[dl]\\\circ&\circ&\circ&\circ}\qquad
\xymatrix@R=10mm@C=0.1mm{\circ\ar@{.}[d]&\circ\ar@{-}[drr]&\circ\ar@{-}[d]&\circ\ar@{-}[dll]\\\circ&\circ&\circ&\circ}$$

$$\xymatrix@R=10mm@C=0.1mm{\circ\ar@{-}[dr]&\circ\ar@{-}[dl]&\circ\ar@{.}[d]&\circ\ar@{.}[d]\\\circ&\circ&\circ&\circ}\qquad
\xymatrix@R=10mm@C=0.1mm{\circ\ar@{.}[dr]&\circ\ar@{.}[dl]&\circ\ar@{-}[dr]&\circ\ar@{-}[dl]\\\circ&\circ&\circ&\circ}\qquad
\xymatrix@R=10mm@C=0.1mm{\circ\ar@{-}[dr]&\circ\ar@{-}[dr]&\circ\ar@{-}[dll]&\circ\ar@{.}[d]\\\circ&\circ&\circ&\circ}\qquad 
\xymatrix@R=10mm@C=0.1mm{\circ\ar@{.}[dr]&\circ\ar@{.}[dr]&\circ\ar@{-}[dr]&\circ\ar@{-}[dlll]\\\circ&\circ&\circ&\circ}\qquad
\xymatrix@R=10mm@C=0.1mm{\bullet\ar@{-}[dr]&\bullet\ar@{-}[drr]&\bullet\ar@{-}[dll]&\bullet\ar@{-}[dl]\\\bullet&\bullet&\bullet&\bullet}\qquad
\xymatrix@R=10mm@C=0.1mm{\circ\ar@{-}[dr]&\circ\ar@{.}[drr]&\circ\ar@{.}[d]&\circ\ar@{-}[dlll]\\\circ&\circ&\circ&\circ}$$

$$\xymatrix@R=10mm@C=0.1mm{\circ\ar@{-}[drr]&\circ\ar@{-}[dl]&\circ\ar@{-}[dl]&\circ\ar@{.}[d]\\\circ&\circ&\circ&\circ}\qquad
\xymatrix@R=10mm@C=0.1mm{\bullet\ar@{-}[drr]&\bullet\ar@{-}[dl]&\bullet\ar@{-}[dr]&\bullet\ar@{-}[dll]\\\bullet&\bullet&\bullet&\bullet}\qquad
\xymatrix@R=10mm@C=0.1mm{\circ\ar@{-}[drr]&\circ\ar@{-}[d]&\circ\ar@{-}[dll]&\circ\ar@{.}[d]\\\circ&\circ&\circ&\circ}\qquad 
\xymatrix@R=10mm@C=0.1mm{\circ\ar@{.}[drr]&\circ\ar@{.}[d]&\circ\ar@{-}[dr]&\circ\ar@{-}[dlll]\\\circ&\circ&\circ&\circ}\qquad
\xymatrix@R=10mm@C=0.1mm{\bullet\ar@{-}[drr]&\bullet\ar@{-}[drr]&\bullet\ar@{-}[dll]&\bullet\ar@{-}[dll]\\\bullet&\bullet&\bullet&\bullet}\qquad
\xymatrix@R=10mm@C=0.1mm{\circ\ar@{.}[drr]&\circ\ar@{.}[drr]&\circ\ar@{-}[dl]&\circ\ar@{-}[dlll]\\\circ&\circ&\circ&\circ}$$

$$\xymatrix@R=10mm@C=0.1mm{\circ\ar@{-}[drrr]&\circ\ar@{.}[dl]&\circ\ar@{.}[dl]&\circ\ar@{-}[dl]\\\circ&\circ&\circ&\circ}\qquad
\xymatrix@R=10mm@C=0.1mm{\circ\ar@{-}[drrr]&\circ\ar@{-}[dl]&\circ\ar@{.}[d]&\circ\ar@{.}[dll]\\\circ&\circ&\circ&\circ}\qquad
\xymatrix@R=10mm@C=0.1mm{\circ\ar@{-}[drrr]&\circ\ar@{.}[d]&\circ\ar@{.}[dll]&\circ\ar@{-}[dl]\\\circ&\circ&\circ&\circ}\qquad 
\xymatrix@R=10mm@C=0.1mm{\circ\ar@{-}[drrr]&\circ\ar@{.}[d]&\circ\ar@{.}[d]&\circ\ar@{-}[dlll]\\\circ&\circ&\circ&\circ}\qquad
\xymatrix@R=10mm@C=0.1mm{\circ\ar@{-}[drrr]&\circ\ar@{-}[dr]&\circ\ar@{.}[dll]&\circ\ar@{.}[dll]\\\circ&\circ&\circ&\circ}\qquad
\xymatrix@R=10mm@C=0.1mm{\circ\ar@{-}[drrr]&\circ\ar@{.}[dr]&\circ\ar@{.}[dl]&\circ\ar@{-}[dlll]\\\circ&\circ&\circ&\circ}$$

Here the dotted lines correspond either to outer (left or right) strings, or to pairs of adjacent strings, and our claim is that all these dotted strings can be deleted. Indeed, for outer strings, this follows from the following computation, by summing over $a$:
$$aX=aY\implies a^*aX=a^*aY\implies X=Y$$
$$Xa=Ya\implies Xaa^*=Yaa^*\implies X=Y$$

As for the adjacent string claim, this follows from a similar computation:
$$XabY=ZabT\implies Xaa^*Y=Zaa^*T\implies XY=ZT$$
$$XabY=ZbaT\implies Xaa^*Y=Za^*aT\implies XY=ZT$$

Now since all the diagrams containing dotted strings correspond to depth 3 spheres, we have just to study the diagrams having no dotted strings. And there are 3 such diagrams, namely those having solid circles, with the corresponding relations being as follows:
$$abcd=cadb,\qquad abcd=bdac,\qquad abcd=cdab$$ 

The first two relations are equivalent, the corresponding diagrams being related by upside-down turning, and produce the usual sphere $\dot{S}^{N-1}_\mathbb K$. Indeed, we have:
$$abcd=cadb\implies abcd=cadb=dcba\implies abb^*d=db^*ba\implies ad=da$$

As for the last relations, these produce the sphere $\dot{S}^{N-1}_{\mathbb K,*}$, because we have:
$$abc=cba\implies abcd=cbad=cdab$$
$$abcd=cdab\implies abcde=cdabe=cbeda\implies abb^*de=b^*beda\implies ade=eda$$

Thus, we have no new monomial sphere at depth 4, as claimed.
\end{proof}

We conjecture that the 10 monomial spheres in Theorem 2.5 are the only ones, regardless of the depth. Solving this conjecture would of course fully clarify our axiomatization.

\section{Unitary quantum groups}

In this section we construct 10 compact quantum groups. We will show later on, in sections 5-6 below, that these are the quantum isometry groups of our 10 spheres.

We use the formalism of compact matrix quantum groups, developed by Woronowicz in \cite{wo1}, \cite{wo2}. For a detailed presentation of the theory, we refer to \cite{ntu}.

We begin with the following key definition, due to Wang \cite{wa1}:

\begin{definition}
The compact quantum groups $O_N^+,U_N^+$ are defined by
\begin{eqnarray*}
C(O_N^+)&=&C^*\left((u_{ij})_{i,j=1,\ldots,N}\Big|u_{ij}=u_{ij}^*,u^t=u^{-1}\right)\\
C(U_N^+)&=&C^*\left((u_{ij})_{i,j=1,\ldots,N}\Big|u^*=u^{-1},u^t=\bar{u}^{-1}\right)
\end{eqnarray*}
with Hopf algebra maps $\Delta(u_{ij})=\sum_ku_{ik}\otimes u_{kj}$, $\varepsilon(u_{ij})=\delta_{ij}$, $S(u_{ij})=u_{ji}^*$.
\end{definition}

As shown in \cite{wa1}, the above two algebras satisfy the axioms of Woronowicz in \cite{wo1}, \cite{wo2}, so the underlying spaces $O_N^+,U_N^+$ are indeed compact quantum groups. We have proper embeddings $O_N\subset O_N^+,U_N\subset U_N^+$, at any $N\geq2$. See \cite{ntu}, \cite{wa1}.

We have as well the following key examples, coming from \cite{bsp}, \cite{bdu}:

\begin{definition}
The half-liberations of $O_N,U_N$ are defined as
\begin{eqnarray*}
C(O_N^*)&=&C(O_N^+)\Big/\left<abc=cba,\forall a,b,c\in\{u_{ij}\}\right>\\
C(U_N^{**})&=&C(U_N^+)\Big/\left<abc=cba,\forall a,b,c\in\{u_{ij},u_{ij}^*\}\right>
\end{eqnarray*}
with Hopf algebra maps $\Delta,\varepsilon,S$ obtained by restriction.
\end{definition}

We refer to \cite{bsp}, \cite{bve} for details regarding $O_N^*$, and to \cite{bdu} for details regarding $U_N^{**}$. As already mentioned, and known since \cite{bdd}, in the unitary case the half-liberation operation is not unique. We will be back to more complicated examples in section 6 below.

Now let us twist the $2+2$ classical and half-classical quantum groups. We agree that all objects to be constructed appear by definition as subspaces of $O_N^+,U_N^+$, obtained by imposing certain extra relations on the standard coordinates $u_{ij}$. We first have:

\begin{proposition}
We have quantum groups $\bar{O}_N\subset O_N^+$, $\bar{U}_N\subset U_N^+$ defined via
$$\alpha\beta=\begin{cases}
-\beta\alpha&{\rm for}\ a,b\in\{u_{ij}\}\ {\rm distinct,\ on\ the\ same\ row\ or\ column}\\
\beta\alpha&{\rm otherwise}
\end{cases}$$
with the usual convention $\alpha=a,a^*$ and $\beta=b,b^*$.
\end{proposition}

\begin{proof}
These quantum groups are well-known, see \cite{bbc}. The idea indeed is that the existence of $\varepsilon,S$ is clear. Regarding now $\Delta$, set $U_{ij}=\sum_ku_{ik}\otimes u_{kj}$. For $j\neq k$ we have:
\begin{eqnarray*}
U_{ij}U_{ik}
&=&\sum_{s\neq t}u_{is}u_{it}\otimes u_{sj}u_{tk}+\sum_su_{is}u_{is}\otimes u_{sj}u_{sk}\\
&=&\sum_{s\neq t}-u_{it}u_{is}\otimes u_{tk}u_{sj}+\sum_su_{is}u_{is}\otimes(-u_{sk}u_{sj})\\
&=&-U_{ik}U_{ij}
\end{eqnarray*}

Also, for $i\neq k,j\neq l$ we have:
\begin{eqnarray*}
U_{ij}U_{kl}
&=&\sum_{s\neq t}u_{is}u_{kt}\otimes u_{sj}u_{tl}+\sum_su_{is}u_{ks}\otimes u_{sj}u_{sl}\\
&=&\sum_{s\neq t}u_{kt}u_{is}\otimes u_{tl}u_{sj}+\sum_s(-u_{ks}u_{is})\otimes(-u_{sl}u_{sj})\\
&=&U_{kl}U_{ij}
\end{eqnarray*}

This finishes the proof in the real case. In the complex case the remaining relations can be checked in a similar way, by putting $*$ exponents in the middle.
\end{proof}

It remains to twist the half-liberated quantum groups $O_N^*,U_N^{**}$. In order to do so, given three coordinates $a,b,c\in\{u_{ij}\}$, let us set $span(a,b,c)=(r,c)$, where $r,c\in\{1,2,3\}$ are the number of rows and columns spanned by $a,b,c$. In other words, if we write $a=u_{ij},b=u_{kl},c=u_{pq}$ then $r=\#\{i,k,p\}$ and $l=\#\{j,l,q\}$. We have then:

\begin{proposition}
We have quantum groups $\bar{O}_N^*\subset O_N^+$, $\bar{U}_N^{**}\subset U_N^+$ defined via
$$\alpha\beta\gamma=\begin{cases}
-\gamma\beta\alpha&{\rm for}\ a,b,c\in\{u_{ij}\}\ {\rm with}\ span(a,b,c)=(\leq 2,3)\ {\rm or}\ (3,\leq 2)\\
\gamma\beta\alpha&{\rm otherwise}
\end{cases}$$
with the usual conventions $\alpha=a,a^*$, $\beta=b,b^*$ and $\gamma=c,c^*$.
\end{proposition}

\begin{proof}
The commutation/anticommutation signs in the statement are as follows:
$$\begin{matrix}
r\backslash c&1&2&3\\
1&+&+&-\\
2&+&+&-\\
3&-&-&+
\end{matrix}$$

We first prove the result for $\bar{O}_N^*$. The construction of the counit, $\varepsilon(u_{ij})=\delta_{ij}$, requires the Kronecker symbols $\delta_{ij}$ to commute/anticommute according to the above table. Equivalently, we must prove that the situation $\delta_{ij}\delta_{kl}\delta_{pq}=1$ can appear only in a case where the above table indicates ``+''. But this is clear, because $\delta_{ij}\delta_{kl}\delta_{pq}=1$ implies $r=c$.

The construction of the antipode $S$ is clear too, because this requires the choice of our $\pm$ signs to be invariant under transposition, and this is true, the table being symmetric. We are therefore left with the construction of $\Delta$. With $U_{ij}=\sum_ku_{ik}\otimes u_{kj}$, we have:
\begin{eqnarray*}
U_{ia}U_{jb}U_{kc}
&=&\sum_{xyz}u_{ix}u_{jy}u_{kz}\otimes u_{xa}u_{yb}u_{zc}\\
&=&\sum_{xyz}\pm u_{kz}u_{jy}u_{ix}\otimes\pm u_{zc}u_{yb}u_{xa}\\
&=&\pm U_{kc}U_{jb}U_{ia}
\end{eqnarray*}

We must prove that, when examining the precise two $\pm$ signs in the middle formula, their product produces the correct $\pm$ sign at the end. The point now is that both these signs depend only on $s=span(x,y,z)$, and for $s=1,2,3$ respectively:

-- For a $(3,1)$ span we obtain $+-$, $+-$, $-+$, so a product $-$ as needed.

-- For a $(2,1)$ span we obtain $++$, $++$, $--$, so a product $+$ as needed.

-- For a $(3,3)$ span we obtain $--$, $--$, $++$, so a product $+$ as needed.

-- For a $(3,2)$ span we obtain $+-$, $+-$, $-+$, so a product $-$ as needed.

-- For a $(2,2)$ span we obtain $++$, $++$, $--$, so a product $+$ as needed.

Together with the fact that our problem is invariant under $(r,c)\to(c,r)$, and with the fact that for a $(1,1)$ span there is nothing to prove, this finishes the proof.
 
For $\bar{U}_N^{**}$ the proof is similar, by putting $*$ exponents in the middle.
\end{proof}

Regarding the inclusions between these quantum groups, we have:

\begin{proposition}
We have the following diagram of quantum groups,
$$\xymatrix@R=15mm@C=15mm{
U_N\ar[r]&U_N^{**}\ar[r]&U_N^+&\bar{U}_N^{**}\ar@.[l]&\bar{U}_N\ar@.[l]\\
O_N\ar[r]\ar[u]&O_N^*\ar[r]\ar[u]&O_N^+\ar[u]&\bar{O}_N^*\ar@.[u]\ar@.[l]&\bar{O}_N\ar@.[u]\ar@.[l]}$$
with all inclusions being proper at $N\geq3$.
\end{proposition}

\begin{proof}
The inclusions are clear, as in the proof of Proposition 1.5. For the  properness assertion, we first compute the classical versions. Our claim is that these are as follows, with the 6 compact groups at right being different at $N\geq 3$:
$$\xymatrix@R=15mm@C=15mm{
U_N\ar[r]&U_N\ar[r]&U_N&Y_N\ar@.[l]&K_N\ar@.[l]\\
O_N\ar[r]\ar[u]&O_N\ar[r]\ar[u]&O_N\ar[u]&X_N\ar@.[u]\ar@.[l]&H_N\ar@.[u]\ar@.[l]}$$

Indeed, regarding the groups $H_N=O_N\cap\bar{O}_N$ and $K_N=U_N\cap\bar{U}_N$, these appear respectively from $O_N,U_N$ by assuming that the standard coordinates satisfy the relations $ab=0$, for any $a\neq b$ on the same row or the same column of $u$. We recognize here the hyperoctahedral group $H_N=\mathbb Z_2\wr S_N$, and its complex version $K_N=\mathbb T\wr S_N$.

Regarding now $X_N,Y_N$, these are certain compact groups, appearing respectively from $O_N,U_N$ by assuming that the coordinates satisfy $abc=0$, under the span conditions producing anticommutation in Proposition 3.4. Since these groups are different, and different as well from $H_N,K_N$ at $N\geq3$, this finishes the proof of our claim.

We deduce that the inclusions on the right in the statement are all proper. As for the properness of the inclusions on the left, this is well-known from \cite{bgo}, \cite{bve}.
\end{proof}

At $N=2$ the situation is similar to the one for the spheres, the diagram of inclusions between the 10 quantum groups being:
$$\xymatrix@R=15mm@C=15mm{
U_2\ar[r]&U_2^{**}\ar[r]&U_2^+&\bar{U}_2^{**}\ar@.[l]&\bar{U}_2\ar@.[l]\\
O_2\ar[r]\ar[u]&O_2^+\ar[r]\ar[u]&O_2^+\ar[u]&O_2^+\ar@.[u]\ar@.[l]&\bar{O}_N\ar@.[u]\ar@.[l]}$$

This can be indeed deduced by using the same arguments as in the sphere case.

Regarding now the relation with our 10 spheres, let us first recall:

\begin{definition}
A quantum group action $G\curvearrowright X$ consists in having a morphism of $C^*$-algebras $\Phi:C(X)\to C(G)\otimes C(X)$ satisfying the following conditions:
\begin{enumerate}
\item Coassociativity: $(id\otimes\Phi)\Phi=(\Delta\otimes id)\Phi$.

\item Counitality: $(\varepsilon\otimes id)\Phi=id$.

\item Faithfulness: $({\rm Im}\,\Phi)(C(G)\otimes 1)$ is dense in $C(G)\otimes C(X)$.
\end{enumerate}
\end{definition}

The morphism in the statement is called coaction. See \cite{bgo}, \cite{ntu}.

Consider now one of our 10 quantum groups, denoted $U_N^\times$. We denote by $S^{N-1}_\times$ the corresponding sphere, with the correspondence between quantum groups and spheres being obtained by superposing the diagrams in Proposition 1.5 and Proposition 3.5. 

We denote the spherical coordinates by $z_i$, in both the real and complex cases.

We have the following result, that we will further improve in section 5 below:

\begin{theorem}
We have an action $U_N^\times\curvearrowright S^{N-1}_\times$, with the corresponding coaction map being given by $\Phi(z_i)=\sum_ju_{ij}\otimes z_j$.
\end{theorem}

\begin{proof}
As a first observation, assuming that the formula $\Phi(z_i)=\sum_ju_{ij}\otimes z_j$ produces indeed a morphism of algebras, the axioms in Definition 3.6 are clear, because they come from the fact that $u=(u_{ij})$ is a fundamental corepresentation for $U_N^\times$. See \cite{bgo}, \cite{sol}.

In order to prove now that we have a morphism of algebras, we must check the fact that the following elements satisfy the defining relations for our spheres:
$$Z_i=\sum_ju_{ij}\otimes z_j$$

We have 10 spheres to be investigated, and the proof goes as follows:

\underline{1-2. $S^{N-1}_{\mathbb R,+},S^{N-1}_{\mathbb C,+}$.} The result for $S^{N-1}_{\mathbb C,+}$ follows from:
\begin{eqnarray*}
\sum_iZ_iZ_i^*&=&\sum_{ijk}(u_{ij}\otimes z_j)(u_{ik}^*\otimes z_k^*)=\sum_{jk}(u^t\bar{u})_{jk}\otimes z_jz_k^*=\sum_j1\otimes z_jz_j^*=1\\
\sum_iZ_i^*Z_i&=&\sum_{ijk}(u_{ik}^*\otimes z_k^*)(u_{ij}\otimes z_j)=\sum_{jk}(u^*u)_{kj}\otimes z_k^*z_j=\sum_j1\otimes z_j^*z_j=1
\end{eqnarray*}

Regarding now $S^{N-1}_{\mathbb R,+}$, the result here follows by restriction, because when assuming $z_i=z_i^*$, the relations $Z_i=Z_i^*$ for any $i$ are equivalent to $u_{ij}=u_{ij}^*$ for any $i,j$.

\underline{3-6. $S^{N-1}_\mathbb R,S^{N-1}_\mathbb C,\bar{S}^{N-1}_\mathbb R,\bar{S}^{N-1}_\mathbb C$.} The results in the classical cases are clear, because the actions in the statement are the usual ones, $O_N\curvearrowright S^{N-1}_\mathbb R$ and $U_N\curvearrowright S^{N-1}_\mathbb C$.

For the sphere $\bar{S}^{N-1}_\mathbb R$ this follows from the following computation, with $i\neq k$:
\begin{eqnarray*}
Z_iZ_k
&=&\sum_{jl}u_{ij}u_{kl}\otimes z_jz_l
=\sum_{j\neq l}u_{ij}u_{kl}\otimes z_jz_l+\sum_ju_{ij}u_{kj}\otimes z_j^2\\
&=&\sum_{j\neq l}u_{kl}u_{ij}\otimes(-z_lz_j)+\sum_j(-u_{kj}u_{ij})\otimes z_j^2\\
&=&-\sum_{jl}u_{kl}u_{ij}\otimes z_lz_j=-Z_kZ_i
\end{eqnarray*}

For the sphere $\bar{S}^{N-1}_\mathbb R$ the proof is similar, by adding $*$ exponents where needed.

\underline{7-10. $S^{N-1}_{\mathbb R,*},S^{N-1}_{\mathbb C,**},\bar{S}^{N-1}_{\mathbb R,*},\bar{S}^{N-1}_{\mathbb C,**}$.} We only prove here the result in the twisted cases, the proof in the untwisted cases being similar, by removing all signs. Let us first discuss the sphere $\bar{S}^{N-1}_{\mathbb R,*}$. We have two sets of conditions to be checked, as follows:

-- For $i,j,k$ distinct, we must have $Z_iZ_jZ_k=-Z_kZ_jZ_i$. We have:
\begin{eqnarray*}
Z_iZ_jZ_k
&=&\sum_{a,b,c\ distinct}u_{ia}u_{jb}u_{kc}\otimes z_az_bz_c+\sum_{a\neq c}u_{ia}u_{ja}u_{kc}\otimes z_az_az_c\\
&&+\sum_{a\neq b}u_{ia}u_{jb}u_{ka}\otimes z_az_bz_a+\sum_{a\neq b}u_{ia}u_{jb}u_{kb}\otimes z_az_bz_b\\
&&+\sum_au_{ia}u_{ja}u_{ka}\otimes z_az_az_a
\end{eqnarray*}

The point now is that we can use the half-commutation relations for both the $u$ and the $z$ variables, and we obtain the formula of $Z_kZ_jZ_i$, with the signs in front of the 5 sums being respectively $+-,-+,-+,-+,-+$. Thus we have $Z_iZ_jZ_k=-Z_kZ_jZ_i$.

-- For $i\neq k$ we must have $Z_iZ_iZ_k=Z_kZ_iZ_i$. We have:
\begin{eqnarray*}
Z_iZ_iZ_k
&=&\sum_{a,b,c\ distinct}u_{ia}u_{ib}u_{kc}\otimes z_az_bz_c+\sum_{a\neq c}u_{ia}u_{ia}u_{kc}\otimes z_az_az_c\\
&&+\sum_{a\neq b}u_{ia}u_{ib}u_{ka}\otimes z_az_bz_a+\sum_{a\neq b}u_{ia}u_{ib}u_{kb}\otimes z_az_bz_b\\
&&+\sum_au_{ia}u_{ia}u_{ka}\otimes z_az_az_a
\end{eqnarray*}

Once again, we can use the half-commutation relations for both the $u$ and the $z$ variables, and we obtain the formula of $Z_kZ_iZ_i$, with the signs in front of the 5 sums being respectively $--,++,++,++,++$. Thus we have $Z_iZ_iZ_k=Z_kZ_iZ_i$.

The proof for $\bar{S}^{N-1}_{\mathbb C,**}$ is similar, by adding $*$ exponents where needed.
\end{proof}

Summarizing, the 10 quantum groups that we have constructed here act on the 10 spheres constructed in section 1. Improving Theorem 3.7, with a universality result for the actions constructed there, will be our main goal in what follows.

\section{Schur-Weyl duality}

In order to get more insight into the structure of our 10 spheres, and into the structure of the actions constructed in Theorem 3.7 above, we need a number of new ingredients, and notably the Schur-Weyl theory for the 10 quantum groups. 

As in \cite{bsp}, we use several types of partitions, as follows:

\begin{definition}
We let $P(k,l)$ be the set of partitions between an upper row of $k$ points and a lower row of $l$ points, and consider the following subsets of $P(k,l)$:
\begin{enumerate}
\item $P_2(k,l)\subset P_{even}(k,l)$: the pairings, and the partitions with blocks having even size.

\item $NC_2(k,l)\subset NC_{even}(k,l)\subset NC(k,l)$: the subsets of noncrossing partitions. 

\item $Perm(k,k)\subset P_2(k,k)$: the pairings having only up-to-down strings.
\end{enumerate}
\end{definition}

Observe that the elements of $Perm(k,k)$ correspond to the permutations in $S_k$, with the usual convention that the permutation diagrams act downwards. See \cite{bsp}, \cite{bve}. 

Given a partition $\tau\in P(k,l)$, we call ``switch'' the operation which consists in switching two neighbors, belonging to different blocks, either in the upper row, or in the lower row. By performing a number of such switches, we can always transform $\tau$ into a certain noncrossing partition $\tau'\in NC(k,l)$, having the same block structure as $\tau$.

We will need the following standard result, regarding the behavior of this switching operation, in the particular case of the partitions having even blocks:

\begin{proposition}
There is a signature map $\varepsilon:P_{even}\to\{-1,1\}$, given by $\varepsilon(\tau)=(-1)^c$, where $c$ is the number of switches needed to make $\tau$ noncrossing. In addition:
\begin{enumerate}
\item For $\tau\in Perm(k,k)$, this is the usual signature.

\item For $\tau\in P_2$ we have $(-1)^c$, where $c$ is the number of crossings.

\item For $\tau\leq\pi\in NC_{even}$, the signature is $1$.
\end{enumerate}
\end{proposition}

\begin{proof}
In order to show that $\varepsilon$ is well-defined, we must prove that the number $c$ in the statement is well-defined modulo 2. It is enough to perform the verification for the noncrossing partitions. More precisely, given $\tau,\tau'\in NC_{even}$ having the same block structure, we must prove that the number of switches $c$ required for the passage $\tau\to\tau'$ is even.

In order to do so, observe that any partition $\tau\in P(k,l)$ can be put in ``standard form'', by ordering its blocks according to the appearence of the first leg in each block, counting clockwise from top left, and then by performing the switches as for block 1 to be at left, then for block 2 to be at left, and so on. Here the required switches are also uniquely determined, by the order coming from counting clockwise from top left. 

Here is an example of such an algorithmic switching operation, with block 1 being first put at left, by using two switches, then with block 2 left unchanged, and then with block 3 being put at left as well, but at right of blocks 1 and 2, with one switch:
$$
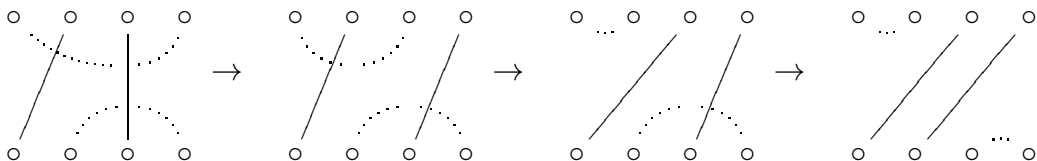
\xymatrix@R=3mm@C=3mm{\circ\ar@/_/@{.}[drr]&\circ\ar@{-}[dddl]&\circ\ar@{-}[ddd]&\circ\\
&&\ar@/_/@{.}[ur]&\\
&&\ar@/^/@{.}[dr]&\\
\circ&\circ\ar@/^/@{.}[ur]&\circ&\circ}
\xymatrix@R=4mm@C=1mm{&\\\to\\&\\& }
\xymatrix@R=3mm@C=3mm{\circ\ar@/_/@{.}[dr]&\circ\ar@{-}[dddl]&\circ&\circ\ar@{-}[dddl]\\
&\ar@/_/@{.}[ur]&&\\
&&\ar@/^/@{.}[dr]&\\
\circ&\circ\ar@/^/@{.}[ur]&\circ&\circ}
\xymatrix@R=4mm@C=1mm{&\\\to\\&\\&}
\xymatrix@R=3mm@C=3mm{\circ\ar@/_/@{.}[r]&\circ&\circ\ar@{-}[dddll]&\circ\ar@{-}[dddl]\\
&&&\\
&&\ar@/^/@{.}[dr]&\\
\circ&\circ\ar@/^/@{.}[ur]&\circ&\circ}
\xymatrix@R=4mm@C=1mm{&\\\to\\&\\& }
\xymatrix@R=3mm@C=3mm{\circ\ar@/_/@{.}[r]&\circ&\circ\ar@{-}[dddll]&\circ\ar@{-}[dddll]\\
&&&\\
&&&\\
\circ&\circ&\circ\ar@/^/@{.}[r]&\circ}$$

The point now is that, under the assumption $\tau\in NC_{even}(k,l)$, each of the moves required for putting a leg at left, and hence for putting a whole block at left, requires an even number of switches. Thus, putting $\tau$ is standard form requires an even number of switches. Now given $\tau,\tau'\in NC_{even}$ having the same block structure, the standard form coincides, so the number of switches $c$ required for the passage $\tau\to\tau'$ is indeed even.

Regarding now the remaining assertions, these are all elementary:

(1) For $\tau\in Perm(k,k)$ the standard form is $\tau'=id$, and the passage $\tau\to id$ comes by composing with a number of transpositions, which gives the signature. 

(2) For a general $\tau\in P_2$, the standard form is of type $\tau'=|\ldots|^{\cup\ldots\cup}_{\cap\ldots\cap}$, and the passage $\tau\to\tau'$ requires $c$ mod 2 switches, where $c$ is the number of crossings. 

(3) Assuming that $\tau\in P_{even}$ comes from $\pi\in NC_{even}$ by merging a certain number of blocks, we can prove that the signature is 1 by proceeding by recurrence. Indeed, we can first assume that we have only 3 blocks, and then we can further use a recurrence on the number of legs, until we reach to the situation where the block in the middle, which crosses the merged outer blocks, is a semicircle, and where the result is clear. 
\end{proof}

With the above notion in hand, we can formulate:

\begin{definition}
Associated to a pair-partition $\pi\in P_2(k,l)$ are the linear maps
$$\dot{T}_\pi(e_{i_1}\otimes\ldots\otimes e_{i_k})=\sum_{j_1\ldots j_l}\dot{\delta}_\pi\begin{pmatrix}i_1\ldots i_k\\ j_1\ldots j_l\end{pmatrix}e_{j_1}\otimes\ldots\otimes e_{j_l}$$
where $\dot{\delta}_\pi\in\{-1,0,1\}$ is constructed, in terms of $\tau=\ker(^i_j)$, as follows:
\begin{enumerate}
\item In the untwisted case, we set $\delta=1$ if $\tau\leq\pi$, and $\delta=0$ otherwise.

\item In the twisted case, we set $\bar{\delta}=\varepsilon(\tau)$ if $\tau\leq\pi$, and $\bar{\delta}=0$ otherwise.
\end{enumerate}
\end{definition}

In the untwisted case we recognize here the usual Brauer intertwiners for $O_N$, discussed for instance in \cite{bco}, \cite{bsp}, and whose formula is simply:
$$T_\pi(e_{i_1}\otimes\ldots\otimes e_{i_k})=\sum_{j:\ker(^i_j)\leq\pi}e_{j_1}\otimes\ldots\otimes e_{j_l}$$

In the twisted case the formula is similar, but requiring this time some signs, constructed according to Proposition 4.2 above. More precisely, we have:
$$\bar{T}_\pi(e_{i_1}\otimes\ldots\otimes e_{i_k})=\sum_{\tau\leq\pi}\varepsilon(\tau)\sum_{j:\ker(^i_j)=\tau}e_{j_1}\otimes\ldots\otimes e_{j_l}$$

Let us work now out a few basic examples of such linear maps, which are of particular interest for the considerations to follow:

\begin{proposition}
The linear map associated to the basic crossing is:
$$\bar{T}_{\slash\!\!\!\backslash}(e_i\otimes e_j)
=\begin{cases}
-e_j\otimes e_i&{\rm for}\ i\neq j\\
e_j\otimes e_i&{\rm otherwise}
\end{cases}$$

The linear map associated to the half-liberating permutation is:
$$\bar{T}_{\slash\hskip-1.6mm\backslash\hskip-1.1mm|\hskip0.5mm}(e_i\otimes e_j\otimes e_k)
=\begin{cases}
-e_k\otimes e_j\otimes e_i&{\rm for}\ i,j,k\ {\rm distinct}\\
e_k\otimes e_j\otimes e_i&{\rm otherwise}
\end{cases}$$

Also, for any noncrossing pairing $\pi\in NC_2$, we have $\bar{T}_\pi=T_\pi$.
\end{proposition}

\begin{proof}
We have to compute the signature of the various partitions involved, and we can use here (1,2,3) in Proposition 4.2. We make the convention that the strings which cross and which are of the same type (e.g. dotted) correspond to the same block.

Regarding the basic crossing and its collapsed version, the signatures are:
$$\xymatrix@R=10mm@C=5mm{\circ\ar@{-}[dr]&\circ\ar@{.}[dl]\\
\circ&\circ}
\xymatrix@R=4mm@C=1mm{&\\\to -1\\&\\& }\qquad\qquad
\xymatrix@R=10mm@C=5mm{\circ\ar@{-}[dr]&\circ\ar@{-}[dl]\\
\circ&\circ}
\xymatrix@R=4mm@C=1mm{&\\\to 1\\&\\& }$$

But this gives the first formula in the statement. Regarding now the second formula, this follows from the following signature computations, obtained by counting the crossings (in the first case), by switching twice as to put the partition in noncrossing form (in the next 3 cases), and by observing that the partition is noncrossing (in the last case):
$$\xymatrix@R=10mm@C=5mm{\circ\ar@/_/@{-}[drr]&\circ\ar@/^/@{.}[d]&\circ\ar@/_/@{~}[dll]\\
\circ&\circ&\circ}
\xymatrix@R=4mm@C=1mm{&\\\to -1\\&\\& }\qquad\qquad
\xymatrix@R=10mm@C=5mm{\circ\ar@/_/@{-}[drr]&\circ\ar@/^/@{.}[d]&\circ\ar@/_/@{-}[dll]\\
\circ&\circ&\circ}
\xymatrix@R=4mm@C=1mm{&\\\to 1\\&\\& }$$
$$\xymatrix@R=10mm@C=5mm{\circ\ar@/_/@{-}[drr]&\circ\ar@/^/@{-}[d]&\circ\ar@/_/@{~}[dll]\\
\circ&\circ&\circ}
\xymatrix@R=4mm@C=1mm{&\\\to 1\\&\\& }\qquad\qquad
\xymatrix@R=10mm@C=5mm{\circ\ar@/_/@{-}[drr]&\circ\ar@/^/@{.}[d]&\circ\ar@/_/@{.}[dll]\\
\circ&\circ&\circ}
\xymatrix@R=4mm@C=1mm{&\\\to 1\\&\\& }\qquad\qquad
\xymatrix@R=10mm@C=5mm{\circ\ar@/_/@{-}[drr]&\circ\ar@/^/@{-}[d]&\circ\ar@/_/@{-}[dll]\\
\circ&\circ&\circ}
\xymatrix@R=4mm@C=1mm{&\\\to 1\\&\\& }$$

Finally, the last assertion follows from Proposition 4.2 (3).
\end{proof}

The relation with the 10 quantum groups comes from:

\begin{proposition}
For an orthogonal quantum group $G$, the following hold:
\begin{enumerate}
\item $\dot{T}_{\slash\!\!\!\backslash}\in End(u^{\otimes 2})$ precisely when $G\subset\dot{O}_N$.

\item $\dot{T}_{\slash\hskip-1.6mm\backslash\hskip-1.1mm|\hskip0.5mm}\in End(u^{\otimes 3})$ precisely when $G\subset\dot{O}_N^*$.
\end{enumerate}
\end{proposition}

\begin{proof}
These results are well-known in the untwisted case, see \cite{bsp}, \cite{bve}.

(1) By using the formula of $\bar{T}_{\slash\!\!\!\backslash}$ in Proposition 4.4, we obtain:
$$(\bar{T}_{\slash\!\!\!\backslash}\otimes1)u^{\otimes 2}(e_i\otimes e_j\otimes1)=\sum_ke_k\otimes e_k\otimes u_{ki}u_{kj}-\sum_{k\neq l}e_l\otimes e_k\otimes u_{ki}u_{lj}$$
$$u^{\otimes 2}(\bar{T}_{\slash\!\!\!\backslash}\otimes1)(e_i\otimes e_j\otimes1)=\begin{cases}
\sum_{kl}e_l\otimes e_k\otimes u_{li}u_{ki}&{\rm if}\ i=j\\
-\sum_{kl}e_l\otimes e_k\otimes u_{lj}u_{ki}&{\rm if}\ i\neq j
\end{cases}$$

For $i=j$ the conditions are $u_{ki}^2=u_{ki}^2$ for any $k$, and $u_{ki}u_{li}=-u_{li}u_{ki}$ for any $k\neq l$. For $i\neq j$ the conditions are $u_{ki}u_{kj}=-u_{kj}u_{ki}$ for any $k$, and $u_{ki}u_{lj}=u_{lj}u_{ki}$ for any $k\neq l$. Thus we have exactly the relations between the coordinates of $\bar{O}_N$, and we are done.

(2) By using the formula of $\bar{T}_{\slash\hskip-1.6mm\backslash\hskip-1.1mm|\hskip0.5mm}$ in Proposition 4.4, we obtain:
\begin{eqnarray*}
(\bar{T}_{\slash\hskip-1.6mm\backslash\hskip-1.1mm|\hskip0.5mm}\otimes1)u^{\otimes 2}(e_i\otimes e_j\otimes e_k\otimes1)
&=&\sum_{abc\ not\ distinct}e_c\otimes e_b\otimes e_a\otimes u_{ai}u_{bj}u_{ck}\\
&-&\sum_{a,b,c\ distinct}e_c\otimes e_b\otimes e_a\otimes u_{ai}u_{bj}u_{ck}
\end{eqnarray*}

On the other hand, we have as well the following formula:
\begin{eqnarray*}
&&u^{\otimes 2}(\bar{T}_{\slash\hskip-1.6mm\backslash\hskip-1.1mm|\hskip0.5mm}\otimes1)(e_i\otimes e_j\otimes e_k\otimes1)\\
&&\ \ \ \ \ \ =\begin{cases}
\sum_{abc}e_c\otimes e_b\otimes e_a\otimes u_{ck}u_{bj}u_{ai}&{\rm for}\ i,j,k\ {\rm not\ distinct}\\
-\sum_{abc}e_c\otimes e_b\otimes e_a\otimes u_{ck}u_{bj}u_{ai}&{\rm for}\ i,j,k\ {\rm distinct}
\end{cases}
\end{eqnarray*}

For $i,j,k$ not distinct the conditions are $u_{ai}u_{bj}u_{ck}=u_{ck}u_{bj}u_{ai}$ for $a,b,c$ not distinct, and $u_{ai}u_{bj}u_{ck}=-u_{ck}u_{bj}u_{ai}$ for $a,b,c$ distinct. For $i,j,k$ distinct the conditions are $u_{ai}u_{bj}u_{ck}=-u_{ck}u_{bj}u_{ai}$ for $a,b,c$ not distinct, and $u_{ai}u_{bj}u_{ck}=u_{ck}u_{bj}u_{ai}$ for $a,b,c$ distinct. Thus we have exactly the relations between the coordinates of $\bar{O}_N^*$, and we are done.
\end{proof}

We prove now that the usual categorical operations on the linear maps $\dot{T}_\pi$, namely the composition, tensor product and conjugation, are compatible with the usual categorical operations on the partitions from \cite{bsp}, namely the composition $(\pi,\sigma)\to[^\sigma_\pi]$, the horizontal concatenation $(\pi,\sigma)\to[\pi\sigma]$, and the upside-down turning $\pi\to\pi^*$. We have:

\begin{proposition}
The assignement $\pi\to\dot{T}_\pi$ is categorical, in the sense that
$$\dot{T}_\pi\otimes\dot{T}_\sigma=\dot{T}_{[\pi\sigma]},\qquad\dot{T}_\pi \dot{T}_\sigma=N^{c(\pi,\sigma)}\dot{T}_{[^\sigma_\pi]},\qquad\dot{T}_\pi^*=\dot{T}_{\pi^*}$$
where $c(\pi,\sigma)$ are certain positive integers.
\end{proposition}

\begin{proof}
By using the definition of $\pi\to\dot{T}_\pi$, we just have to understand the behaviour of the generalized Kronecker symbol construction $\pi\to\dot{\delta}_\pi$, under the various categorical operations on the partitions $\pi$. We have to check three conditions, as follows:

\underline{1. Concatenation.} In the untwisted case, this follows from the following formula:
$$\delta_\pi\begin{pmatrix}i_1\ldots i_p\\ j_1\ldots j_q\end{pmatrix}
\delta_\sigma\begin{pmatrix}k_1\ldots k_r\\ l_1\ldots l_s\end{pmatrix}
=\delta_{[\pi\sigma]}\begin{pmatrix}i_1\ldots i_p&k_1\ldots k_r\\ j_1\ldots j_q&l_1\ldots l_s\end{pmatrix}$$

In the twisted case, it is enough to check the following formula:
$$\varepsilon\left(\ker\begin{pmatrix}i_1\ldots i_p\\ j_1\ldots j_q\end{pmatrix}\right)
\varepsilon\left(\ker\begin{pmatrix}k_1\ldots k_r\\ l_1\ldots l_s\end{pmatrix}\right)=
\varepsilon\left(\ker\begin{pmatrix}i_1\ldots i_p&k_1\ldots k_r\\ j_1\ldots j_q&l_1\ldots l_s\end{pmatrix}\right)$$

Let us denote by $\tau,\nu$ the partitions on the left, so that the partition on the right is of the form $\rho\leq[\tau\nu]$. Now by switching to the noncrossing form, $\tau\to\tau'$ and $\nu\to\nu'$, the partition on the right transforms into $\rho\to\rho'\leq[\tau'\nu']$. Now since $[\tau'\nu']$ is noncrossing, we can use Proposition 4.2 (3), and we obtain the result.

\underline{2. Composition.} In the untwisted case, this follows from the following formula from \cite{bsp}, where $c(\pi,\sigma)$ is the number of closed loops obtained when composing:
$$\sum_{j_1\ldots j_q}\delta_\pi\begin{pmatrix}i_1\ldots i_p\\ j_1\ldots j_q\end{pmatrix}
\delta_\sigma\begin{pmatrix}j_1\ldots j_q\\ k_1\ldots k_r\end{pmatrix}
=N^{c(\pi,\sigma)}\delta_{[^\pi_\sigma]}\begin{pmatrix}i_1\ldots i_p\\ k_1\ldots k_r\end{pmatrix}$$

In order to prove now the result in the twisted case, it is enough to check that the signs match. More precisely, we must establish the following formula:
$$\varepsilon\left(\ker\begin{pmatrix}i_1\ldots i_p\\ j_1\ldots j_q\end{pmatrix}\right)
\varepsilon\left(\ker\begin{pmatrix}j_1\ldots j_q\\ k_1\ldots k_r\end{pmatrix}\right)
=\varepsilon\left(\ker\begin{pmatrix}i_1\ldots i_p\\ k_1\ldots k_r\end{pmatrix}\right)$$

Let $\tau,\nu$ be the partitions on the left, so that the partition on the right is of the form $\rho\leq[^\tau_\nu]$. Our claim is that we can jointly switch $\tau,\nu$ to the noncrossing form. Indeed, we can first switch as for $\ker(j_1\ldots j_q)$ to become noncrossing, and then switch the upper legs of $\tau$, and the lower legs of $\nu$, as for both these partitions to become noncrossing. 

Now observe that when switching in this way to the noncrossing form, $\tau\to\tau'$ and $\nu\to\nu'$, the partition on the right transforms into $\rho\to\rho'\leq[^{\tau'}_{\nu'}]$. Now since $[^{\tau'}_{\nu'}]$ is noncrossing, we can apply Proposition 4.2 (3), and we obtain the result.

\underline{3. Involution.} Here we must prove the following formula:
$$\dot{\delta}_\pi\begin{pmatrix}i_1\ldots i_p\\ j_1\ldots j_q\end{pmatrix}=\dot{\delta}_{\pi^*}\begin{pmatrix}j_1\ldots j_q\\ i_1\ldots i_p\end{pmatrix}$$

But this is clear, both in the untwisted and twisted cases, and we are done.
\end{proof}

In order to formulate the duality result, we use words $\alpha,\beta,\ldots$ over the symbols $u,\bar{u}$. Given such a word $\alpha$, we denote by $u^{\otimes\alpha}$ the corepresentation obtained by performing the corresponding tensor product, by inserting $\otimes$ signs between the $u,\bar{u}$ symbols.
 
Also, we denote by $P_2^*\subset P_2$ the set of pairings having the property that when labelling cyclically the legs $\bullet\circ\bullet\circ\ldots$, each string joins a black leg to a white leg. See \cite{bve}.

With these conventions, the Schur-Weyl duality result is as follows:

\begin{theorem}
We have $Hom(u^{\otimes\alpha},u^{\otimes\beta})=span(\dot{T}_\pi|\pi\in P_\times(\alpha,\beta))$, where the sets of diagrams for the the $10$ quantum groups, with inclusions between them, are
$$\xymatrix@R=12mm@C=12mm{
\mathcal P_2\ar[d]&\mathcal P_2^*\ar[l]\ar[d]&\mathcal N\mathcal C_2\ar[l]\ar[d]\ar@.[r]&\mathcal P_2^*\ar@.[r]\ar@.[d]&\mathcal P_2\ar@.[d]\\
P_2&P_2^*\ar[l]&NC_2\ar[l]\ar@.[r]&P_2^*\ar@.[r]&P_2}$$
with the convention $P_\times(\alpha,\beta)=P_\times(|\alpha|,|\beta|)$, where $|.|$ is the word length, and where the upper subsets $\mathcal P_\times(\alpha,\beta)\subset P_\times(\alpha,\beta)$ consist of partitions with strings joining $u,\bar{u}$.
\end{theorem}

\begin{proof}
In the real untwisted case, all the diagrams are already known, see \cite{bsp}, \cite{bve}. In the complex untwisted case, the proof is similar. In the twisted case now, the result for $\bar{O}_N$ follows as in \cite{bco}, by using Proposition 4.5 (1), Proposition 4.6, and Tannakian duality \cite{wo2}. For the other twisted quantum groups, the result follows by functoriality, as in \cite{bsp}, \cite{bve}, by using Proposition 4.5, and by adding $*$ exponents where needed.
\end{proof}

\section{Affine isometries}

In this section we go back to Theorem 3.7, and improve the result found there. 

It is known from \cite{bg1} that proving universality results for quantum group actions requires a good knowledge of the linear relations satisfied by the various products of coordinates. And we can deal now with such problems, by using Schur-Weyl duality.

We will need the Weingarten integration formula. We begin with:

\begin{definition}
Let $P_2(l)=P_2(0,l)$. For $\pi\in P_2(l)$ we set:
$$\dot{\xi}_\pi=\sum_{j_1\ldots j_l}\dot{\delta}_\pi(j_1\ldots j_l)e_{j_1}\otimes\ldots\otimes e_{j_l}$$
In other words, we denote by $\dot{\xi}_\pi$ the vector $\dot{T}_\pi$ constructed in Definition 4.3.
\end{definition}

In the classical case, we recognize the usual Brauer fixed vectors for $O_N$. In the twisted case, the formula is similar, this time making appear some signatures:
$$\bar{\xi}_\pi=\sum_{\tau\leq\pi}\varepsilon(\tau)\sum_{j:\ker j=\tau}e_{j_1}\otimes\ldots\otimes e_{j_l}$$

Here are a few examples of such vectors, coming from the computations in Proposition 4.4 above. First, the vector associated to the basic crossing is:
$$\bar{\xi}_{\cap\hskip-1.2mm\cap}=\sum_ie_i\otimes e_i\otimes e_i\otimes e_i-\sum_{i\neq j}e_i\otimes e_j\otimes e_i\otimes e_j$$

Also, the vector associated to the half-liberating pairing $(123123)$ is:
$$\bar{\xi}_{\cap\hskip-1.4mm\cap\hskip-1.5mm\cap}
=\sum_{ijk\ not\ distinct}e_i\otimes e_j\otimes e_k\otimes e_i\otimes e_j\otimes e_k-\sum_{i,j,k\ distinct}e_i\otimes e_j\otimes e_k\otimes e_i\otimes e_j\otimes e_k$$

Finally, observe that for any noncrossing pairing $\pi\in NC_2(l)$, we have $\bar{\xi}_\pi=\xi_\pi$.

We will need the following simple fact:

\begin{proposition}
The scalar products between the vectors $\dot{\xi}_\pi$ are given by
$$<\dot{\xi}_\pi,\dot{\xi}_\sigma>=N^{|\pi\vee\sigma|}$$
and hence coincide in the twisted and the untwisted cases.
\end{proposition}

\begin{proof}
In the twisted case, we have the following computation:
\begin{eqnarray*}
<\bar{\xi}_\pi,\bar{\xi}_\sigma>
&=&\left\langle\sum_{j:\ker j\leq\pi}\varepsilon(\ker j)e_{j_1}\otimes\ldots\otimes e_{j_l},\sum_{j:\ker j\leq\sigma}\varepsilon(\ker j)e_{j_1}\otimes\ldots\otimes e_{j_l}\right\rangle\\
&=&\sum_{j:\ker j\leq(\pi\vee\sigma)}\varepsilon(\ker j)^2=\sum_{j:\ker j\leq(\pi\vee\sigma)}1=N^{|\pi\vee\sigma|}
\end{eqnarray*}

In the untwisted case the computation is similar, with the signs dissapearing right from the beginning. Thus, in both cases we obtain the formula in the statement.
\end{proof}

Given one of our quantum groups $\dot{U}_N^\times$, and an exponent vector $\alpha=(\alpha_1,\ldots,\alpha_k)\in\{1,*\}^k$, we denote by $P_\times(\alpha)=P_\times(\emptyset,\alpha)$ the set of pairings found in Theorem 4.7 above for $\dot{U}_N^\times$, having no upper points, and having the lower points labelled by the entries of $\alpha$, according to the identifications $u\to 1,\bar{u}\to *$. With this convention, we have:

\begin{proposition}
We have the Weingarten type formula
$$\int_{\dot{U}_N^\times}u_{i_1j_1}^{\alpha_1}\ldots u_{i_kj_k}^{\alpha_k}=\sum_{\pi,\sigma\in P_\times(\alpha)}\dot{\delta}_\pi(i_1\ldots i_k)\dot{\delta}_\sigma(j_1\ldots j_k)W_{kN}^\alpha(\pi,\sigma)$$
where $W_{kN}^\alpha=(G_{kN}^\alpha)^{-1}$, with $G_{kN}^\alpha(\pi,\sigma)=N^{|\pi\vee\sigma|}$, for $\pi,\sigma\in P_\times(\alpha)$.
\end{proposition}

\begin{proof}
This follows indeed as in \cite{bco}, by using Theorem 4.7 and Proposition 5.2. Observe that the Weingarten matrix is the same in the twisted and the untwisted cases.
\end{proof}

Now back to the spheres, we first have the following result:

\begin{lemma}
The linear relations satisfied by the variables $r_{ij}=z_iz_j$ are as follows:
\begin{enumerate}
\item For $S^{N-1}_\mathbb R,\bar{S}^{N-1}_\mathbb R$ we have $r_{ij}=\pm r_{ji}$, and no other relations.

\item For the remaining $8$ spheres, these elements are linearly independent.
\end{enumerate}
In addition, a similar result holds for the variables $c_{ij}=z_iz_j^*$.
\end{lemma}

\begin{proof}
We first prove the assertion regarding the variables $r_{ij}=z_iz_j$. We have 10 spheres to be investigated, and the proof goes as follows:

\underline{1-2. $S^{N-1}_\mathbb R,S^{N-1}_\mathbb C$.} The results here are clear.

\underline{3-4. $\bar{S}^{N-1}_\mathbb R,\bar{S}^{N-1}_\mathbb C$.} We prove first the result for $\bar{S}^{N-1}_\mathbb R$. We use the model $z_i\to Z_i=u_{1i}$, where $u_{ij}$ are the standard coordinates on $\bar{O}_N$. We have:
$$<Z_iZ_j,Z_kZ_l>
=\int_{\bar{O}_N}u_{1i}u_{1j}u_{1l}u_{1k}=\sum_{\pi,\sigma\in P_2(4)}\bar{\delta}_\sigma(i,j,l,k)W_{4N}(\pi,\sigma)$$

Since $P_2(4)=\{\cap\cap,\Cap,\cap\!\!\cap\}$, the Weingarten matrix on the right is given by:
$$W_{4N}=\begin{pmatrix}N^2&N&N\\ N&N^2&N\\ N&N&N^2\end{pmatrix}^{-1}=\frac{1}{N(N-1)(N+2)}\begin{pmatrix}N+1&-1&-1\\ -1&N+1&-1\\ -1&-1&N+1\end{pmatrix}$$

We conclude that we have the following formula:
$$<Z_iZ_j,Z_kZ_l>=\frac{1}{N(N+2)}\sum_{\sigma\in P_2(4)}\bar{\delta}_\sigma(i,j,l,k)$$

The matrix on the right, taken with indices $i\leq j$ and $k\leq l$, is then invertible. Thus the variables $Z_iZ_j$ are linearly independent, and so must be the variables $z_iz_j$.

For the sphere $\bar{S}^{N-1}_\mathbb C$, a similar computation, using now a $\bar{U}_N$ model, gives:
$$<Z_iZ_j,Z_kZ_l>=\int_{\bar{U}_N}u_{1i}u_{1j}u_{1l}^*u_{1k}^*=\sum_{\pi,\sigma\in P_2(11**)}\bar{\delta}_\sigma(i,j,l,k)W_{4N}^{11**}(\pi,\sigma)$$

We have $P_2(11**)=\{\Cap,\cap\!\!\cap\}$, and the corresponding Weingarten matrix is:
$$W_{4N}^{11**}=\begin{pmatrix}N^2&N\\ N&N^2\end{pmatrix}^{-1}=\frac{1}{N(N^2-1)}\begin{pmatrix}N&-1\\ -1&N\end{pmatrix}$$

We therefore obtain the following formula:
$$<Z_iZ_j,Z_kZ_l>=\frac{1}{N(N+1)}\sum_{\sigma\in P_2(11**)}\bar{\delta}_\sigma(i,j,l,k)$$

Once again, since the matrix on the right is invertible, we obtain the result.

\underline{5-6. $S^{N-1}_{\mathbb R,*},\bar{S}^{N-1}_{\mathbb R,*}$.} We can use here a $2\times2$ matrix trick from \cite{bdu}. Consider indeed one of the spheres $S^{N-1}_\mathbb C/\bar{S}^{N-1}_\mathbb C$, with coordinates denoted $y_1,\ldots,y_N$, and let us set:
$$Z_i=\begin{pmatrix}0&y_i\\ y_i^*&0\end{pmatrix}$$

As explained in the proof of Theorem 1.7 above, these matrices produce models for $S^{N-1}_{\mathbb R,*},\bar{S}^{N-1}_{\mathbb R,*}$. Now observe that the elements $r_{ij}=z_iz_j$ map in this way to:
$$R_{ij}=Z_iZ_j=\begin{pmatrix}0&y_i\\ y_i^*&0\end{pmatrix}\begin{pmatrix}0&y_j\\ y_j^*&0\end{pmatrix}=\begin{pmatrix}y_iy_j^*&0\\ 0&y_i^*y_j\end{pmatrix}$$

Thus, the result follows from the result for $\bar{S}^{N-1}_\mathbb R,\bar{S}^{N-1}_\mathbb C$, established above.

\underline{7-10. $S^{N-1}_{\mathbb R,+},S^{N-1}_{\mathbb C,+},S^{N-1}_{\mathbb C,**},\bar{S}^{N-1}_{\mathbb C,**}$.} The results here follow simply by functoriality, from those established above, for the smaller spheres $S^{N-1}_{\mathbb R,*},\bar{S}^{N-1}_{\mathbb R,*}$.

Finally, the proof of the last assertion is similar, with no new computations needed in the real case, where $r_{ij}=c_{ij}$, and with the same Weingarten matrix, this time coming from the set $P_2(1*1*)=\{\cap\cap,\Cap\}$, appearing in the complex case.
\end{proof}

We can improve now Theorem 3.7 above. First, we have:

\begin{proposition}
$O_N,U_N,\bar{O}_N,\bar{U}_N,O_N^+,U_N^+$ are the biggest compact quantum groups acting on their respective spheres, with the actions leaving $span(z_i)$ invariant.
\end{proposition}

\begin{proof}
The fact that $span(z_i)$ is left invariant means that the coaction must be of the form $\Phi(z_i)=\sum_ju_{ij}\otimes z_j$. We have six situations to be investigated, as follows:

\underline{1-2. $S^{N-1}_{\mathbb R,+},S^{N-1}_{\mathbb C,+}$.} Let us go back to the proof of Theorem 3.7. As explained there, the coaction axioms are equivalent to the fact that $u=(u_{ij})$ is a corepresentation. Also, with the notation $Z_i=\sum_ju_{ij}\otimes z_j$, we know from there that we have:
$$\sum_iZ_iZ_i^*=\sum_{jk}(u^t\bar{u})_{jk}\otimes z_jz_k^*,\qquad\sum_iZ_i^*Z_i=\sum_{jk}(u^*u)_{kj}\otimes z_k^*z_j$$

Now by using Lemma 5.4 above for the free spheres, we deduce that the conditions $\sum_iZ_iZ_i^*=\sum_iZ_i^*Z_i=1$ are equivalent to the conditions $u^t\bar{u}=u^*u=1$. 

Now since $u$ is already known to be a corepresentation, by the results of Woronowicz in \cite{wo1} it follows that $u$ must be a biunitary corepresentation, and we are done.

\underline{3-4. $S^{N-1}_\mathbb R,S^{N-1}_\mathbb C$.} For the sphere $S^{N-1}_\mathbb R$ this is done in \cite{bg1}. We reproduce here the proof, in view of some further extensions and modifications. First, we have:
\begin{eqnarray*}
\Phi(z_iz_j)
&=&\sum_{kl}u_{ik}u_{jl}\otimes z_kz_l\\
&=&\sum_ku_{ik}u_{jk}\otimes z_k^2+\sum_{k<l}(u_{ik}u_{jl}+u_{il}u_{jk})\otimes z_kz_l\\
&=&\sum_{k\leq l}(u_{ik}u_{jl}+u_{il}u_{jk})\otimes\left(1-\frac{\delta_{kl}}{2}\right)z_kz_l
\end{eqnarray*}

We deduce from this that $\Phi$ maps the commutators $[z_i,z_j]$ as follows:
\begin{eqnarray*}
\Phi([z_i,z_j])
&=&\sum_{k\leq l}(u_{ik}u_{jl}+u_{il}u_{jk}-u_{jk}u_{il}-u_{jl}u_{ik})\otimes\left(1-\frac{\delta_{kl}}{2}\right)z_kz_l\\
&=&\sum_{k\leq l}\left([u_{ik},u_{jl}]-[u_{jk},u_{il}]\right)\otimes\left(1-\frac{\delta_{kl}}{2}\right)z_kz_l
\end{eqnarray*}

Now since the variables $\{z_kz_l|k\leq l\}$ are linearly independent, we obtain from this $[u_{ik},u_{jl}]=[u_{jk},u_{il}]$, for any $i,j,k,l$. Moreover, if we apply now the antipode we further obtain $[u_{lj},u_{ki}]=[u_{li},u_{kj}]$, and by relabelling, $[u_{ik},u_{jl}]=[u_{il},u_{jk}]$. We therefore conclude that we have $[u_{ik},u_{jl}]=0$ for any $i,j,k,l$, and this finishes the proof. See \cite{bg1}.

For $S^{N-1}_\mathbb C$ the beginning of the proof is similar, and gives $[u_{ik},u_{jl}]=[u_{jk},u_{il}]$. Now if we apply the antipode followed by the involution we obtain as before $[u_{lj},u_{ki}]=[u_{li},u_{kj}]$, then $[u_{ik},u_{jl}]=[u_{il},u_{jk}]$, and finally $[u_{ik},u_{jl}]=0$. Thus the coordinates are subject to the commutation relations $ab=ba$, and by using a standard categorial trick, mentioned before Definition 2.2 above, we have as well $ab^*=b^*a$, and we are done.

\underline{5-6. $\bar{S}^{N-1}_\mathbb R$, $\bar{S}^{N-1}_\mathbb C$.} The proof here is similar to the above proof for $S^{N-1}_\mathbb R,S^{N-1}_\mathbb C$, by using Lemma 5.4, and by adding signs where needed. First, for $\bar{S}^{N-1}_\mathbb R$ we have:
$$\Phi(z_iz_j)=\sum_ku_{ik}u_{jk}\otimes z_k^2+\sum_{k<l}(u_{ik}u_{jl}-u_{il}u_{jk})\otimes z_kz_l$$

We deduce that with $[[a,b]]=ab+ba$ we have the following formula:
$$\Phi([[z_i,z_j]])=\sum_k[[u_{ik},u_{jk}]]\otimes z_k^2+\sum_{k<l}([u_{ik},u_{jl}]-[u_{il},u_{jk}])\otimes z_kz_l$$

Now assuming $i\neq j$, we have $[[z_i,z_j]]=0$, and we therefore obtain $[[u_{ik},u_{jk}]]=0$ for any $k$, and $[u_{ik},u_{jl}]=[u_{il},u_{jk}]$ for any $k<l$. By applying the antipode and then by relabelling, the latter relation gives $[u_{ik},u_{jl}]=0$, and we are done.

The proof for $\bar{S}^{N-1}_\mathbb C$ is similar, by using the above-mentioned categorical trick, in order to deduce from the relations $ab=\pm ba$ the remaining relations $ab^*=\pm b^*a$.
\end{proof}

In order to deal with the half-liberated cases, we will need:

\begin{lemma}
Consider one of the spheres $S^{N-1}_{\mathbb R,*},S^{N-1}_{\mathbb C,**},\bar{S}^{N-1}_{\mathbb R,*},\bar{S}^{N-1}_{\mathbb C,**}$.
\begin{enumerate}
\item The variables $z_az_bz_c$ with $a<c$ and $a,b,c$ distinct are linearly independent.

\item These variables are independent as well from any $z_az_bz_c$ with $a,b,c$ not distinct.
\end{enumerate}
In addition, a similar result holds for the variables of type $z_az_b^*z_c$.
\end{lemma}

\begin{proof}
We use the same method as in the proof of Lemma 5.4, with models coming from the quantum groups $O_N^*,U_N^{**},\bar{O}_N^*,\bar{U}_N^{**}$. For the quantum groups $O_N^*,\bar{O}_N^*$, we have:
$$<Z_aZ_bZ_c,Z_iZ_jZ_k>=\int_{\dot{O}_N^*}u_{1a}u_{1b}u_{1c}u_{1k}u_{1j}u_{1i}=\sum_{\pi,\sigma\in P_2^*(6)}\bar{\delta}_\sigma(a,b,c,k,j,i)W_{6N}(\pi,\sigma)$$

The set $P_2^*(6)\simeq P_2^*(3,3)$ is by definition formed by the following pairings:
$$
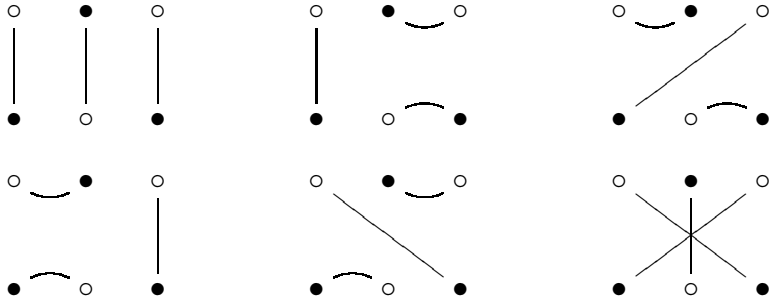
\xymatrix@R=10mm@C=5mm{\circ\ar@{-}[d]&\bullet\ar@{-}[d]&\circ\ar@{-}[d]\\ \bullet&\circ&\bullet}\qquad\qquad
\xymatrix@R=10mm@C=5mm{\circ\ar@{-}[d]&\bullet\ar@/_/@{-}[r]&\circ\\ \bullet&\circ\ar@/^/@{-}[r]&\bullet}\qquad\qquad
\xymatrix@R=10mm@C=5mm{\circ\ar@/_/@{-}[r]&\bullet&\circ\ar@{-}[dll]\\ \bullet&\circ\ar@/^/@{-}[r]&\bullet}$$
$$\xymatrix@R=10mm@C=5mm{\circ\ar@/_/@{-}[r]&\bullet&\circ\ar@{-}[d]\\ \bullet\ar@/^/@{-}[r]&\circ&\bullet}\qquad\qquad
\xymatrix@R=10mm@C=5mm{\circ\ar@{-}[drr]&\bullet\ar@/_/@{-}[r]&\circ\\ \bullet\ar@/^/@{-}[r]&\circ&\bullet}\qquad\qquad
\xymatrix@R=10mm@C=5mm{\circ\ar@{-}[drr]&\bullet\ar@{-}[d]&\circ\ar@{-}[dll]\\ \bullet&\circ&\bullet}$$

Now observe that the scalar products of each of these pairings with all the 6 pairings are always, up to a permutation of the terms, $N^3,N^2,N^2,N^2,N,N$. Thus the Gram matrix is stochastic, $G_{6N}\xi=\xi$, where $\xi=(1,\ldots,1)^t$ is the all-one vector. Thus we have $W_{6N}\xi=W_{6N}G_{6N}\xi=\xi$, and so the Weingarten matrix is stochastic too. We conclude that, up to a universal constant depending only on $N$, we have:
$$<Z_aZ_bZ_c,Z_iZ_jZ_k>\sim\sum_{\sigma\in P_2^*(6)}\bar{\delta}_\sigma(a,b,c,k,j,i)$$

Now by computing the rank of this matrix, we obtain the result.

Regarding now the last assertion, this follows from the same computation. Indeed, comparing the products of type $Z_aZ_b^*Z_c$ leads to the same formula and conclusion, because the pairings in $P_2^*(6)$ are all compatible with the leg labelling $1*1*1\,*$.
\end{proof}

We have now all ingredients for fully improving Theorem 3.7 above, with the remark that we will further process this result in section 6 below:

\begin{theorem}
Each quantum group $U_N^\times$ is the biggest compact quantum group acting on its respective sphere $S^{N-1}_\times$, with the action leaving $span(z_i)$ invariant.
\end{theorem}

\begin{proof}
In view of Proposition 5.5, we just have to discuss the 4 half-liberated cases. 

The idea here will be that for the spheres $S^{N-1}_{\mathbb R,*},S^{N-1}_{\mathbb C,**},\bar{S}^{N-1}_{\mathbb R,*},\bar{S}^{N-1}_{\mathbb C,**}$ the proof will be similar to the one for the spheres $S^{N-1}_\mathbb R,S^{N-1}_\mathbb C,\bar{S}^{N-1}_\mathbb R,\bar{S}^{N-1}_\mathbb C$, by replacing the commutators $[u_{ia},u_{kc}]=u_{ia}u_{kc}-u_{kc}u_{ia}$ by quantities of type $[u_{ia},u_{jb},u_{kc}]=u_{ia}u_{jb}u_{kc}-u_{kc}u_{jb}u_{ia}$.

We only discuss the twisted case, the proof in the untwisted case being similar. For a coaction on $\bar{S}^{N-1}_{\mathbb R,*}$, we have two sets of conditions to be verified, as follows:

-- For $i,j,k$ distinct, we must have $Z_iZ_jZ_k=-Z_kZ_jZ_i$. We have:
\begin{eqnarray*}
Z_iZ_jZ_k
&=&\sum_{a,b,c\ distinct}u_{ia}u_{jb}u_{kc}\otimes z_az_bz_c+\sum_{a\neq c}u_{ia}u_{ja}u_{kc}\otimes z_az_az_c\\
&&+\sum_{a\neq b}u_{ia}u_{jb}u_{ka}\otimes z_az_bz_a+\sum_{a\neq b}u_{ia}u_{jb}u_{kb}\otimes z_az_bz_b\\
&&+\sum_au_{ia}u_{ja}u_{ka}\otimes z_az_az_a
\end{eqnarray*}

Now by using Lemma 5.6, all three sums appearing at left must vanish, and the 2 sums on the right must add up to 0 too. From the vanishing of the first sum we conclude, by proceeding as in the proof of Proposition 5.5, that the coordinates $u_{ia}$ satisfy the relations $abc=cba$, when their span is $(3,3)$. Similarly, from the vanishing of the other sums we obtain $abc=-cba$ for a $(3,2)$ span, and $abc=-cba$ for a $(3,1)$ span.

-- For $i\neq k$ we must have $Z_iZ_iZ_k=Z_kZ_iZ_i$. We have:
\begin{eqnarray*}
Z_iZ_iZ_k
&=&\sum_{a,b,c\ distinct}u_{ia}u_{ib}u_{kc}\otimes z_az_bz_c+\sum_{a\neq c}u_{ia}u_{ia}u_{kc}\otimes z_az_az_c\\
&&+\sum_{a\neq b}u_{ia}u_{ib}u_{ka}\otimes z_az_bz_a+\sum_{a\neq b}u_{ia}u_{ib}u_{kb}\otimes z_az_bz_b\\
&&+\sum_au_{ia}u_{ia}u_{ka}\otimes z_az_az_a
\end{eqnarray*}

From the first sum we get $abc=-cba$ for a $(3,2)$ span, from the next three sums we get $abc=cba$ for a $(2,2)$ span, and from the last sum we get $abc=cba$ for a $(2,1)$ span.

Since we have as well, trivially, $abc=cba$ for a $(1,1)$ span, we have reached to the defining relations for the quantum group $\bar{O}_N^*$, and we are done.

Finally, the proof for the sphere $\bar{S}^{N-1}_{\mathbb C,**}$ is similar, by adding $*$ exponents in the middle, and by using the last assertion in Lemma 5.6.
\end{proof}

Observe that Theorem 5.7 is a quantum isometry group computation, in the affine sense of \cite{go2}. More precisely, if we define the affine actions on the real/complex spheres to be the actions of closed subgroups $G\subset O_N^+/U_N^+$ given by coaction maps of type $\Phi(z_i)=\sum_ju_{ij}\otimes z_j$, then Theorem 5.7 computes the corresponding quantum isometry groups. We refer to \cite{ch1}, \cite{ch2}, \cite{go2}, \cite{qsa} for more details regarding the affine action formalism.

\section{Further results, conclusion}

We discuss in this section a number of further topics, including the construction and basic properties of the integration functional for our 10 spheres, the Riemannian aspects of these spheres, and a proposal for an extended formalism, comprising 18 spheres.

In order to construct the integration, we use the associated quantum group:

\begin{definition}
Given one of the spheres $S^{N-1}_\times$, we denote by $U_N^\times$ the associated quantum group, and we let $R_N^\times\subset C(U_N^\times)$ be the subalgebra generated by $u_{11},\ldots,u_{1N}$. 
\end{definition}

By the universal property of $C(S^{N-1}_\times)$ we have a morphism $\pi:C(S^{N-1}_\times)\to R_N^\times$ mapping $x_i\to u_{1i}$, and by composing with the restriction of the Haar functional $I:C(U_N^\times)\to\mathbb C$, we obtain a trace $tr:C(S^{N-1}_\times)\to\mathbb C$. In order to prove that $tr$ is ergodic, we use:

\begin{lemma}
The following formula holds, over the sphere $\dot{S}^{N-1}_\times$,
$$\sum_{j_1\ldots j_l}\dot{\delta}_\pi(j_1,\ldots,j_l)z_{j_1}^{\alpha_1}\ldots z_{j_l}^{\alpha_l}=1$$
for any exponent vector $\alpha=(\alpha_1,\ldots,\alpha_k)\in\{1,*\}^k$, and any pairing $\pi\in P_\times(\alpha)$.
\end{lemma}

\begin{proof}
In the untwisted case this was proved in \cite{bgo}. Let us discuss now the case of $\bar{S}^{N-1}_\mathbb R$. By switching as for putting $\pi$ in standard form, $\pi'=\sqcap\ldots\sqcap$, we obtain:
\begin{eqnarray*}
\sum_{j_1\ldots j_l}\bar{\delta}_\pi(j_1,\ldots,j_l)z_{j_1}\ldots z_{j_l}
&=&\sum_{j:\ker j\leq\pi}\varepsilon(\ker j)z_{j_1}\ldots z_{j_l}\\
&=&\sum_{j':\ker j'\leq\pi'}z_{j_1'}\ldots z_{j_l'}\\
&=&\sum_{j_1'j_3'\ldots j_{l-1}'}z_{j_1'}^2\ldots z_{j_{l-1}'}^2=1
\end{eqnarray*}

For the sphere $\bar{S}^{N-1}_\mathbb C$ the proof is similar, with the last equality coming this time from $\sum_jz_jz_j^*=\sum_jz_j^*z_j=1$. Finally, in the half-liberated cases the proof is similar as well, by using $abc\to cba$ switches as in \cite{bgo}, and in the free cases the result is clear.
\end{proof}

Now back to the trace constructed above, we have here:

\begin{proposition}
Consider the trace $tr:C(S^{N-1}_\times)\to\mathbb C$ obtained by composing the canonical surjection onto the first row algebra of $U_N^\times$ with the Haar functional.
\begin{enumerate}
\item $tr$ is invariant, $(id\otimes tr)\Phi(x)=tr(x)1$.

\item $tr$ is ergodic, $(I\otimes id)\Phi=tr(.)1$.

\item $tr$ is the unique positive unital invariant trace on $C(S^{N-1}_\times)$.
\end{enumerate}
\end{proposition}

\begin{proof}
We use a general method from \cite{bgo}, which was further developed in \cite{bss}. The idea is that the result will follow by using the Weingarten integration formula:

(1) This is clear, by using the invariance of the Haar integral of $C(U_N^\times)$.

(2) It is enough to check the equality on a product $z_{i_1}^{\alpha_1}\ldots z_{i_k}^{\alpha_k}$. The left term is:
\begin{eqnarray*}
(I\otimes id)\Phi(z_{i_1}^{\alpha_1}\ldots z_{i_k}^{\alpha_k})
&=&\sum_{j_1\ldots j_k}I(u_{i_1j_1}^{\alpha_1}\ldots u_{i_kj_k}^{\alpha_k})z_{j_1}^{\alpha_1}\ldots z_{j_k}^{\alpha_k}\\
&=&\sum_{j_1\ldots j_k}\sum_{\pi,\sigma\in P_\times(\alpha)}\dot{\delta}_\pi(i)\dot{\delta}_\sigma(j)W_{kN}^\alpha(\pi,\sigma)z_{j_1}^{\alpha_1}\ldots z_{j_k}^{\alpha_k}\\
&=&\sum_{\pi,\sigma\in P_\times(\alpha)}\dot{\delta}_\pi(i)W_{kN}^\alpha(\pi,\sigma)\sum_{j_1\ldots j_k}\dot{\delta}_\sigma(j)z_{j_1}^{\alpha_1}\ldots z_{j_k}^{\alpha_k}
\end{eqnarray*}

By using Lemma 6.2 the sum on the right is 1, so we get:
$$(I\otimes id)\Phi(z_{i_1}^{\alpha_1}\ldots z_{i_k}^{\alpha_k})=\sum_{\pi,\sigma\in P_\times(\alpha)}\dot{\delta}_\pi(i)W_{kN}^\alpha(\pi,\sigma)1$$

On the other hand, another application of the Weingarten formula gives:
\begin{eqnarray*}
tr(z_{i_1}^{\alpha_1}\ldots z_{i_k}^{\alpha_k})1
&=&I(u_{1i_1}^{\alpha_1}\ldots u_{1i_k}^{\alpha_k})1\\
&=&\sum_{\pi,\sigma\in P_\times(\alpha)}\dot{\delta}_\pi(1)\dot{\delta}_\sigma(i)W_{kN}^\alpha(\pi,\sigma)1\\
&=&\sum_{\pi,\sigma\in P_\times(\alpha)}\dot{\delta}_\sigma(i)W_{kN}^\alpha(\pi,\sigma)1
\end{eqnarray*}

Since the Weingarten function is symmetric in $\pi,\sigma$, this finishes the proof. 

(3) Let $\tau:C(S^{N-1}_\times)\to\mathbb C$ be a trace satisfying the invariance condition. We have:
$$\tau (I\otimes id)\Phi(x)
=(I\otimes\tau)\Phi(x)
=I(id\otimes\tau)\Phi(x)
=I(\tau(x)1)
=\tau(x)$$

On the other hand, according to the formula in (2) above, we have as well:
$$\tau (I\otimes id)\Phi(x)=\tau(tr(x)1)=tr(x)$$

Thus we obtain $\tau=tr$, which finishes the proof.
\end{proof}

As a consequence, we have the following result:

\begin{proposition}
The following algebras, with generators and traces, are isomorphic, when replaced with their GNS completions with respect to their canonical traces:
\begin{enumerate}
\item The algebra $C(S^{N-1}_\times)$, with generators $z_1,\ldots,z_N$.

\item The row algebra $R_N^\times\subset C(U_N^\times)$ generated by $u_{11},\ldots,u_{1N}$.
\end{enumerate}
\end{proposition}

\begin{proof}
Consider the quotient map $\pi:C(S^{N-1}_\times)\to R_N^\times$, used in the proof of Proposition 6.2. The invariance property of the integration functional $I:C(U_N^\times)\to\mathbb C$ shows that $tr'=I\pi$ satisfies the invariance condition in Proposition 6.2, so we have $tr=tr'$. Together with the positivity of $tr$ and with the basic properties of the GNS construction, this shows that $\pi$ induces an isomorphism at the level of GNS algebras, as claimed.
\end{proof}

We make now the following convention, for the reminder of this paper: 

\begin{definition}
We agree from now on to replace each algebra $C(S^{N-1}_\times)$ with its GNS completion with respect to the canonical trace, coming from $U_N^\times$.
\end{definition}

As a first observation, the classical spheres $S^{N-1}_\mathbb R,S^{N-1}_\mathbb C$ are left unchanged by this modification, because the trace comes from the usual uniform measure on them. The free spheres $S^{N-1}_{\mathbb R,+},S^{N-1}_{\mathbb C,+}$ are however ``cut'' by this construction, for instance because this happens at the quantum group level, since $O_N^+,U_N^+$ are not coamenable. See \cite{ntu}.

Regarding the various half-liberations and twists, here we do not know. The faithfulness question for the trace of $S^{N-1}_{\mathbb R,*}$, which was raised some time ago in \cite{bgo}, is still open.

As in \cite{bgo}, we can now construct spectral triples for our spheres, in some weak sense. The idea is that we have inclusions $\dot{S}^{N-1}_\mathbb R\subset\dot{S}^{N-1}_\times\subset S^{N-1}_{\mathbb C,+}$, and so we have surjective maps $C(S^{N-1}_{\mathbb C,+})\to C(\dot{S}^{N-1}_\times)\to C(\dot{S}^{N-1}_\mathbb R)$, and we can construct the Laplacian filtration as projection/pullback of the Laplacian filtration for $S^{N-1}_{\mathbb C,+}/\dot{S}^{N-1}_\mathbb R$.

More precisely, we have the following construction:

\begin{definition}
Associated to each sphere $S^{N-1}_\times$ is the spectral triple $(A,H,D)$, where $A=C(S^{N-1}_\times)$, the dense subalgebra ${\mathcal A}$ is the linear span of all finite words in the generators $z_i,z_i^*$, and the operator $D$ acting on $H=L^2(A,tr)$ is defined as follows:
\begin{enumerate}
\item Set $H_k=span(z_{i_1}^{\alpha_1}\ldots z_{i_r}^{\alpha_r}|r\leq k,\alpha\in\{1,*\}^r)$.

\item Define $E_k=H_k\cap H_{k-1}^\perp$, so that we have $H=\oplus_{k=0}^\infty E_k$.

\item Finally, set $Dx=\lambda_kx$, for any $x\in E_k$, where $\lambda_k$ are distinct numbers.
\end{enumerate} 
\end{definition}

As pointed out in \cite{bgo}, it is quite unclear what the correct eigenvalues should be. In the various half-liberated cases the problem can be probably approached by using the geometry of the associated projective planes \cite{ddl}. In the free cases the situation seems to require the use of advanced analytic techniques, like those in \cite{cfk}, \cite{dgo}.

This type of issue is in fact well-known in the quantum group context, for noncommutative manifolds constructed by using various liberation procedures. See \cite{bvz}.

Without precise eigenvalues, we are in fact in the orthogonal filtration framework of \cite{bsk}, \cite{thi}. As explained there, having such a filtration suffices for constructing a quantum isometry group. In our case, we can formulate the following result:

\begin{theorem}
We have $G^+(S^{N-1}_\times)=O_N^\times$, with the quantum isometry group taken in the spectral triple sense of \cite{go1}, for all the $5$ real spheres.
\end{theorem}

\begin{proof}
This was proved in \cite{bgo} in the untwisted case, and the proof in the twisted case is similar. Consider indeed the standard coaction $\Phi:C(S^{N-1}_\times)\to C(O_N^\times)\otimes C(S^{N-1}_\times)$. This extends to a unitary representation on the GNS space $H_N^\times$, that we denote by $U$. We have $\Phi(H_k)\subset C(O_N^\times)\otimes H_k$, which reads $U(H_k)\subset H_k$. By unitarity we get as well $U(H_k^\perp)\subset H_k^\perp$, so each $E_k$ is $U$-invariant, and $U,D$ must commute. Thus, $\Phi$ is isometric with respect to $D$. Finally, the universality of $O_N^\times$ follows from Theorem 5.7.
\end{proof}

In the complex case the situation is more delicate, and would require a good understanding of the notion of complex affine action, in the noncommutative Riemannian geometry setting. For an exposition of some of the technical difficulties here, see \cite{go2}.

There are of course many other questions regarding our 10 spheres, and their geometry. Besides the two fundamental questions raised above, regarding the faithfulness of the trace on the full algebra, and the construction of the eigenvalues, further interesting questions regard orientability issues, and the existence of a Dirac operator, cf. \cite{bg2}, \cite{co3}.

To summarize, regarding the geometric structure of our spheres, we have so far more questions than answers. We intend to clarify the situation in a future paper.

\bigskip

We would like to discuss now a possible extension of our formalism, from 10 to 18 spheres. The idea is that such an extension should come in three steps, as follows:

{\bf I.} First, the projective planes for the 10 spheres can be computed by using methods from \cite{bgo}, \cite{bve}, by using Schur-Weyl duality. These are as follows, where $P^N_\mathbb R,P^N_\mathbb C,\bar{P}^N_\mathbb R,\bar{P}^N_\mathbb C,P^N_{\mathbb R,+}$ are by definition the projective planes for $S^{N-1}_\mathbb R,S^{N-1}_\mathbb C,\bar{S}^{N-1}_\mathbb R,\bar{S}^{N-1}_\mathbb C,S^{N-1}_{\mathbb R,+}$:
$$\xymatrix@R=16mm@C=16mm{
P^N_\mathbb C\ar[r]&P^N_\mathbb C\ar[r]&P^N_{\mathbb R,+}&\bar{P}^N_\mathbb C\ar@.[l]&\bar{P}^N_\mathbb C\ar@.[l]\\
P^N_\mathbb R\ar[r]\ar[u]&P^N_\mathbb C\ar[r]\ar[u]&P^N_{\mathbb R,+}\ar[u]&\bar{P}^N_\mathbb C\ar@.[u]\ar@.[l]&\bar{P}^N_\mathbb R\ar@.[u]\ar@.[l]}$$

{\bf II.} We recall from \cite{ban} that the free complexification operation amounts in multiplying the standard coordinates by a unitary which is free from them. The free complexifications of the 10 spheres can be computed by using the projective planes and techniques from \cite{ban}, \cite{rau}, the conclusion being that the diagram is as follows, with $\dot{S}^{N-1}_{\mathbb C,*}$ obtained via relations of type $ab^*c=\pm cb^*a$, and $\dot{S}^{N-1}_{\mathbb C,\#}$ obtained via relations of type $ab^*=\pm ba^*,a^*b=\pm b^*a$:
$$\xymatrix@R=15mm@C=15mm{
S^{N-1}_{\mathbb C,*}\ar[r]&S^{N-1}_{\mathbb C,*}\ar[r]&S^{N-1}_{\mathbb C,+}&\bar{S}^{N-1}_{\mathbb C,*}\ar@.[l]&\bar{S}^{N-1}_{\mathbb C,*}\ar@.[l]\\
S^{N-1}_{\mathbb C,\#}\ar[r]\ar[u]&S^{N-1}_{\mathbb C,*}\ar[r]\ar[u]&S^{N-1}_{\mathbb C,+}\ar[u]&\bar{S}^{N-1}_{\mathbb C,*}\ar@.[u]\ar@.[l]&\bar{S}^{N-1}_{\mathbb C,\#}\ar@.[u]\ar@.[l]}$$

{\bf III.} The problem now is that, when adding these 4 new spheres, we will lose the fact that our set of spheres is stable under intersections. More precisely, in order for this to hold, we must add 4 more spheres, namely $\dot{S}^{N-1}_{\mathbb C,\circ}=\dot{S}^{N-1}_{\mathbb C,\#}\cap\dot{S}^{N-1}_{\mathbb C,**}$ and $\dot{S}^{N-1}_{\mathbb C,-}=\dot{S}^{N-1}_{\mathbb C,\circ}\cap\dot{S}^{N-1}_\mathbb C$. The diagram of inclusions between the $18$ spheres is then as follows:
$$\xymatrix@R=4mm@C=5mm{
&&S^{N-1}_{\mathbb C,\#}\ar[dr]&&&&\bar{S}^{N-1}_{\mathbb C,\#}\ar@.[dl]&\\
&S^{N-1}_{\mathbb C,\circ}\ar[dr]\ar[ur]&&S^{N-1}_{\mathbb C,*}\ar[r]&S^{N-1}_{\mathbb C,+}&\bar{S}^{N-1}_{\mathbb C,*}\ar@.[l]&&\bar{S}^{N-1}_{\mathbb C,\circ}\ar@.[ul]\ar@.[dl]\\
S^{N-1}_{\mathbb C,-}\ar[ur]\ar[dr]&&S^{N-1}_{\mathbb C,**}\ar[ur]&&&&\bar{S}^{N-1}_{\mathbb C,**}\ar@.[ul]&&\bar{S}^{N-1}_{\mathbb C,-}\ar@.[ul]\ar@.[dl]\\
&S^{N-1}_\mathbb C\ar[ur]&&&&&&\bar{S}^{N-1}_\mathbb C\ar@.[ul]\\
S^{N-1}_\mathbb R\ar[uu]\ar[rr]&&S^{N-1}_{\mathbb R,*}\ar[rr]\ar[uu]&&S^{N-1}_{\mathbb R,+}\ar[uuu]&&\bar{S}^{N-1}_{\mathbb R,*}\ar@.[ll]\ar@.[uu]&&\bar{S}^{N-1}_\mathbb R\ar@.[ll]\ar@.[uu]}$$

By functoriality, the set of 18 spheres follows to be stable under free complexification.

Regarding the extension of our various results, the first, and main problem, concerns the axiomatization. The complexification formula $z_i=ux_i$ suggests to use diagrams with each leg labelled either $\circ\times$ or $\times\bullet$, with the simplification rules $\circ\bullet\to\emptyset$ and $\bullet\circ\to\emptyset$. We believe that an axiomatization is possible along these lines, and that this should lead to an extension of the other results as well, but we do not have any precise result here.

Regarding some further extensions of our $10+8$ formalism, interesting here, as a technical ingredient, would be to have classification results for the easy quantum groups $U_N\subset G\subset U_N^+$, or more generally for the easy quantum groups $O_N\subset G\subset U_N^+$. In principle, the needed ingredients for dealing with such questions are available from \cite{bsp}, \cite{fre}, \cite{rwe}. In practice, however, it is not clear what the ``19-th sphere'' should be.

As a general conclusion, in the undeformed world we have $10+8$ main geometries. For the simplest such geometry, the one of $\mathbb R^N$, the group $O_N$ appears twice, first as a quantum isometry group, $O_N=G^+(S^{N-1}_\mathbb R)$, and second as a manifold, $O_N\subset S^{N^2-1}_\mathbb R$. The situation is similar in the complex case, and for the remaining $8+8$ geometries as well. With this perspective in mind, several results concerning the subgroups $G\subset O_N$, taken either as groups, or as manifolds, should have extensions to other geometries.

This adds to the various questions raised throughout the paper.

\end{document}